\documentclass[reqno]{amsart}

\usepackage{verbatim}
\usepackage{graphicx}
\usepackage{amssymb}
\usepackage{esint}
\theoremstyle{plain}
\newtheorem{thm}{Theorem}[section]
\newtheorem{lemma}[thm]{Lemma}

\newtheorem{cor}[thm]{Corollary}

\numberwithin{equation}{section}

\newtheorem{rem}[thm]{Remark}

\newcommand{\beq}{\begin{eqnarray}}
\newcommand{\eeq}{\end{eqnarray}}
\newcommand{\beqno}{\begin{eqnarray*}}
\newcommand{\eeqno}{\end{eqnarray*}}

\allowdisplaybreaks

\begin{document}
\title[Compressible nematic liquid crystal flow]
{Global finite energy weak solutions to the compressible nematic liquid crystal flow in dimension three}
%\maketitle
\author[J. Lin]{Junyu Lin}
\address{Department of Mathematics\\
 South China  University of Technology\\
 Guangdong 510640, P. R. China}
\email{scjylin@sctu.edu.cn}
\author[B. Lai]{Baishun Lai}
\address{School of Mathematics and Information Sciences\\
 Henan University\\
 Kaifeng 475004, Henan, P. R. China}
\email{laibaishun@gmail.com}
\author[C. Wang]{Changyou Wang}
\address{
Department of Mathematics, Purdue University\\
150 N. University Street, West Lafayette, IN 47907, USA}
\email{wang2482@purdue.edu}
\date{\today}
\maketitle
\begin{abstract} In this paper,  we consider  the initial and boundary value problem of a simplified
compressible nematic liquid crystal flow in $\Omega\subset\mathbb R^3$. We establish the existence of
global weak solutions, provided the initial orientational director field $d_0$ lies in the hemisphere $\mathbb S^2_+$.
\end{abstract}

%\maketitle
%\footnotesize{\bf Keywords:} Wellposedness; Morrey spaces; Landau-Lifshitz-Gilbert eqiation
%\end{document}

\setcounter{section}{0} \setcounter{equation}{0}
\section{Introduction }

The continuum theory of liquid crystals was developed by Ericksen \cite{JE}  and Leslie \cite{FL}
during the period of 1958 through 1968, see also the book by De Gennes \cite{DG}.
Since then there have been remarkable research developments in liquid crystals from both theoretical and applied aspects.
When the fluid containing nematic liquid crystal materials is at rest, we have the well-known Oseen-Frank theory for static nematic liquid crystals, see Hardt-Lin-Kinderlehrer \cite{HLK} on the analysis of energy minimal configurations of nematic liquid crystals.
In general, the motion of fluid always takes place. The so-called Ericksen-Leslie system is a macroscopic continuum description of the time evolution of the material under influence of both the flow velocity field $u$ and the macroscopic description of the microscopic orientation configurations $d$ of rod-like liquid crystals.

When the fluid is an incompressible, viscous  fluid,
Lin \cite{Lin0} first derived a simplified Ericksen-Leslie system (i.e. $\rho=1$ and ${\rm{div}}  u=0$ in the equation (\ref{1.1}) below) modeling liquid crystal flows in 1989. Subsequently, Lin and Liu \cite{LL, LL1} have made some important analytic studies, such as the global existence of weak and strong solutions and the partial regularity of suitable weak solutions,  of the simplified Ericksen-Leslie system, under the assumption that the liquid crystal director field is of varying length by Leslie's terminology or variable degree of orientation by Ericksen's terminology.  When dealing with the system (\ref{1.1}) with $\rho=1$ and
${\rm{div}} u=0$,  in dimension two Lin-Lin-Wang \cite{LLW}  and Lin-Wang \cite{LinW} have established the existence of a unique global weak solution, that has at most finitely many possible singular time,  for the initial-boundary value problem in bounded domains (see also Hong \cite{Hong}, Xu-Zhang \cite{XZ}, and Lei-Li-Zhang \cite{LLZ} for some related works); and in dimension three Lin-Wang \cite{LW1} have obtained the existence of global weak solutions very recently when the initial director field
$d_0$ maps to the hemisphere $\mathbb S^2_+$.

When the fluid is compressible, the simplified Ericksen-Leslie system (\ref{1.1})
becomes more complicate, which is a strongly coupling system
between the compressible Navier-Stokes equation and the transported harmonic map heat flow to $\mathbb S^2$.
It seems worthwhile to be explored for the  mathematical analysis of (\ref{1.1}).
We would like to mention that there have been both modeling study, see Morro \cite {Morro},
and numerical study, see Zakharov-Vakulenko \cite{ZV}, on the hydrodynamics of compressible
nematic liquid crystals  under the influence of temperature gradient or electromagnetic forces.

Now let's introduce the simplified Ericksen-Leslie system for compressible nematic liquid crystal flow.
Let $\Omega\subset \mathbb R^3$ be a bounded, smooth domain, $\mathbb S^2\subset\mathbb R^3$
be the unit sphere, and $0<T\leq+\infty$. We will consider a simplified version of the three dimensional
hydrodynamic flow of the compressible nematic liquid crystal flow in $\Omega\times (0,T)$,
i.e., $(\rho, u,d):\Omega\times (0,T)\to\mathbb R_+\times\mathbb R^3\times\mathbb S^2$ solves
\beq\label{1.1}
\begin{cases}
\partial_t\rho+\nabla\cdot(\rho u)=0,\\
\partial_t(\rho u)+\nabla\cdot(\rho u\otimes u)+a\nabla\rho^{\gamma}=\mathcal{L}u-\nabla\cdot\big(\nabla d\odot\nabla d-\frac12|\nabla d|^2\mathbb{I}_3\big),\\
\partial_t d+u\cdot\nabla d=\triangle d+|\nabla d|^2d,    \end{cases} \eeq
under the initial and boundary condition:
\beq\label{1.2}
\begin{cases}
\rho(x,0)=\rho_0(x), \ \rho u(x,0)=m_0(x),  \ d(x,0)=d_0(x), \ x\in\Omega,\\
u(x,t)=0,\ \ d(x,t)=d_0(x), \ \ x\in \partial\Omega,\ t>0,
\end{cases}
\eeq
where $\rho:\Omega\times[0,T)\rightarrow \mathbb R_+$ denotes density function of the fluid, $u:\Omega\times[0,T)\rightarrow \mathbb R^3$ denotes velocity field of the fluid,  $d:\Omega\times[0,T)
\rightarrow \mathbb S^2$ denotes direction field of the averaged macroscopic molecular orientations, $\nabla\cdot$ denotes the divergence operator in $\mathbb R^3$,
$\mathbb{I}_3$ is the $3\times3$ identity matrix, $P(\rho)=a\rho^{\gamma}$, with $a>0$ and $\gamma>1$,
denotes the pressure function associated with an isentropic fluid, $\mathcal{L}$ is the Lam\'e operator defined by
$$\mathcal{L}u=\mu\triangle u+(\mu+\lambda)\nabla(\nabla\cdot u),$$
where $\mu$ and $\lambda$ represent the shear viscosity and the bulk viscosity coefficients of the
fluid respectively, which satisfy the natural physical condition:
\begin{equation}
\label{1.3}\mu>0,\ \ \ \widetilde{\mu}:=\mu+\lambda\geq 0,
\end{equation}
 $\nabla d\odot\nabla d$ denotes the $3\times 3$ matrix valued function whose ($i,j$)-entry is
$\langle \partial_{x_i} d, \partial_{x_j}d\rangle$ for $1\le i, j\le 3$, and
$\displaystyle u\otimes u=(u^iu^j)_{1\le i, j\le 3}$.

Throughout this paper, we denote $\mathbb S^2_+=\big\{y=(y^1,y^2,y^3)\in\mathbb S^2: \ y^3\ge 0\big\}$
as the upper hemisphere, $\chi_E$ denote the characteristic function of a set $E\subset\mathbb R^3$,
$$H^1(\Omega, \mathbb S^2)=\Big\{d\in H^1(\Omega,\mathbb R^3):
\ d(x)\in\mathbb S^2 \ {\rm{a.e.}}\ x\in\Omega\Big\},
$$
and $\displaystyle A:B=\sum_{i,j=1}^3A_{ij}B_{ij}$ denotes the scalar product of two $3\times 3$ matrices.
For $0<T\le +\infty$, denote
$$Q_T=\Omega\times (0,T), \ \displaystyle \partial_pQ_T
=(\Omega\times\{0\})\cup(\partial\Omega\times (0,T)),
\ \mathcal D'(Q_T)=(C_0^\infty(Q_T))'.
$$

We say $(\rho, u, d):\Omega\times [0,T)\to \mathbb R_+\times\mathbb R^3\times \mathbb S^2$ is a
{\it finite energy} weak solution of the initial-boundary value problem (\ref{1.1})-(\ref{1.2}) if the following
properties hold:
\begin{itemize}
\item[(i)] $\rho\geq0,\ \rho\in L^{\infty}((0,T), L^\gamma(\Omega)),
\ u\in L^2((0,T), H^1(\Omega, \mathbb R^3))$, and $d\in L^2((0,T), H^1(\Omega,\mathbb S^2))$.
\item[(ii)] the system (\ref{1.1}) holds in $\mathcal D'(Q_T)$, $(\ref{1.1})_1$ also holds in $\mathcal D'(\mathbb R^3\times (0,T))$
provided $(\rho, u)$ is prolonged by zero in $\mathbb R^3\setminus\Omega$,
$(\rho,\rho u, d)(x,0)=(\rho_0(x), m_0(x), d_0(x))$  for a.e. $x\in\Omega$, and $(u,d)(x,t)=(0, d_0(x))$ on $\partial\Omega\times (0,T)$
in the sense of traces.
\item[(iii)] $(\rho, u)$ satisfies $(\ref{1.1})_1$ in the sense of the renormalized solutions
introduced by DiPerna-Lions \cite{L1}, that is, $(\rho,u)$ satisfies
\begin{equation}
\label{1.7} \partial_t\big(b(\rho)\big)+\nabla\cdot(b(\rho)u)+
\big(b'(\rho)\rho-b(\rho)\big)\nabla\cdot u=0,
\end{equation}
 in the sense of distributions in $\mathbb R^3\times (0,+\infty)$ for any $b\in C^1((0,+\infty))\cap C([0,+\infty))$ such that
\beq\label{1.8}b'(z)=0\ \mbox{for\ all}\ z\in (0,+\infty) \ \mbox{large\ enough,} \ \mbox{say}\ z\geq M,
\eeq
where the constant $M>0$ may vary for different functions $b$'s. Here
$(\rho,u)$ is prolonged by zero outside $\Omega$.
\item [(iv)] $(\rho, u,d)$ satisfies the following energy inequality
\begin{equation}
\label{1.9}{\bf E}(t)+
\int_0^t\int_{\Omega}\big(\mu|\nabla u|^2+\widetilde{\mu}|\nabla\cdot u|^2+|\triangle d+|\nabla d|^2d|^2\big)\leq
{\bf E}(0),
\end{equation}
for almost all $0<t<T$. Here
\begin{equation}\label{energyt}
{\bf E}(t):=\int_{\Omega}\Big(\frac12\rho|u|^2+\frac{a\rho^{\gamma}}{\gamma-1}+\frac12|\nabla d|^2\Big)(t)
\end{equation}
is the total energy of $(\rho, u,d)$ at time $t>0$, and
\begin{equation}
\label{energy0}
{\bf E}(0)=\int_{\Omega}\Big(\frac {|m_0|^2}{2\rho_0}\chi_{\{\rho_0\ge 0\}}+
\frac{a\rho_0^{\gamma}}{\gamma-1}+\frac12|\nabla d_0|^2\Big)
\end{equation}
is the initial energy.
\end{itemize}

There have been some earlier results on (\ref{1.1}). In dimension one, the existence of global strong solutions and weak
solutions to (\ref{1.1}) has been obtained by \cite{DLWW} and \cite{DWW} respectively. In dimension two, the existence of
global weak solution of (\ref{1.1}), under the condition that the image of $d_0$ is contained in $\mathbb S^2_+$,
was obtained by \cite{JSW}. In dimension three, the local existence of strong solutions of (\ref{1.1}) has been studied by
\cite{HWW1} and \cite{HWW2}.  The compressible limit of compressible nematic liquid crystal flow (\ref{1.1}) has been
studied by \cite{DHXWZ}. We also mention a related work \cite{LZZ}. When considering the compressible nematic liquid crystal flow (\ref{1.1}) under the assumption that the director $d$ has variable degree of orientations, the global existence of weak solutions in dimension three has been obtained by  \cite{LQ} and \cite{WY1} respectively.

In this paper, we are mainly interested in  the existence of finite energy weak solutions of (\ref{1.1})-(\ref{1.2}) in dimension three.
Our main states as follows.
\begin{thm}\label{th:1.1} Assume $\gamma>\frac32$ and the condition (\ref{1.3}) holds. If the initial data $(\rho_0,m_0,d_0)$ satisfies the following condition:
\begin{equation}
\label{1.4}0\le \rho_0\in L^{\gamma}(\Omega),
\end{equation}
\begin{equation}
\label{1.5}m_0\in L^{\frac{2\gamma}{\gamma+1}}(\Omega),\ \ m_0\chi_{\{\rho_0=0\}}=0,\ \ \frac{|m_0|^2}{\rho_0}\chi_{\{\rho_0>0\}}\in L^1(\Omega),
\end{equation}
and
\begin{equation}
\label{1.6}d_0\in H^1(\Omega,\mathbb S^2), \ \mbox{with}\ d_0(x)\in\mathbb S_+^2
\ \mbox{a.e.}\ x\in \Omega.
\end{equation}
Then there exists a global finite energy weak solution $(\rho,u,d):\Omega\times [0,+\infty)
\to\mathbb R_+\times \mathbb R^3\times \mathbb S^2$ to the initial and boundary
value problem (\ref{1.1})-(\ref{1.2}) such that
\begin{itemize}
\item[(i)] $d=(d^1,d^2,d^3)\in L^{\infty}((0,+\infty), H^1(\Omega,\mathbb S^2))$
and $d^3(x,t)\geq0$ a.e. $(x,t)\in\Omega\times(0,+\infty)$.
\item [(ii)] it holds
\begin{equation}\label{stationary}
\int_0^\infty \eta(t)\int_\Omega \Big(\nabla d\odot\nabla d-\frac12|\nabla d|^2\mathbb I_3\Big): \nabla X
+\int_0^\infty \eta(t)\int_\Omega \big\langle \partial_t d+u\cdot\nabla d, X\cdot\nabla d\big\rangle=0,
\end{equation}
for any  $X\in C_0^1(\Omega,\mathbb R^3)$ and $\eta\in C_0^1((0,+\infty))$.
\end{itemize}
\end{thm}

The main ideas of proof of Theorem 1.1 rely on (i) the precompactness results, due to Lin-Wang \cite{LW1},
on approximated Ginzburg-Landau equations $\{d_\epsilon\}$ with bounded energies, bounded $L^2$-tension
fields, and the condition $|d_\epsilon|\le 1$ and $d_\epsilon^3\ge -1+\delta$ for $\delta>0$, and
(ii) suitable adaption of compactness properties of renormalized solutions of compressible Navier-Stokes
equations established by Lions \cite{L1} and Feireisl and his collaborators \cite{F1}, \cite{FP}, and
\cite{FNP1}.

For any global finite energy weak solutions of (\ref{1.1}) and (\ref{1.2}) that satisfies the properties stated in
Theorem \ref{th:1.1}, we are able to establish the following
preliminary result on its large time asymptotic behavior.

\begin{cor}\label{long-time} Under the same assumptions of Theorem \ref{th:1.1},
let $(\rho, u, d):\Omega\times [0,+\infty)\to\mathbb R_+\times \mathbb R^3\times
\mathbb S^2$ be {\rm any} global finite energy weak solution of (\ref{1.1}) and (\ref{1.2}) that
satisfies the properties of Theorem \ref{th:1.1}.
Then there exist $t_n\rightarrow\infty$ and a harmonic map
$d_\infty\in H^1\cap C^\infty(\Omega, \mathbb S^2_+)$, with $d_\infty=d_0$
on $\partial\Omega$, such that
\begin{equation}\label{uniform_limit}
\Big(\rho(\cdot, t_n), u(\cdot, t_n), d(\cdot, t_n)\Big)\rightarrow \Big(\rho_{0,\infty}, 0, d_\infty\Big)
\ {\rm{in}}\ L^\gamma(\Omega)\times L^{p}(\Omega)\times H^1_{\rm{loc}}(\Omega),
\end{equation}
for any $1<p<6$, where $\rho_{0,\infty}:=\displaystyle\frac{1}{|\Omega|}\int_\Omega \rho_0>0$
is the average of the initial mass.
\end{cor}
\begin{rem} {\rm It is a very interesting question to ask whether the convergence in (\ref{uniform_limit}) holds for
$t\rightarrow +\infty$. We plan to address it in a future work. We would like to point out
that such a property has been established by \cite{FP} for the compressible Navier-Stokes equation.
For the compressible flow of nematic liquid crystals with variable degree of orientations,  see Wang-Yu
\cite{WY1} for the large time asymptotic behavior of global weak solutions.}

{\rm The paper is written as follows. In section 2, we provide some preliminary estimates of (\ref{1.1}).
In section 3, we briefly review a compactness theorem due to Lin and Wang \cite{LW1}. In section 4,
we review the main results by Wang-Yu \cite{WY1} on nematic liquid crystal flows with variable lengths
of directors. In section 5, we prove Theorem 1.1. In section 6, we prove Corollary 1.2.}

\end{rem}

\setcounter{section}{1}
\setcounter{equation}{0}
\section{Global energy inequality and estimates based on the maximum principle}

In this section, we will provide several basic properties of the hydrodynamic flow of
compressible nematic liquid crystals (\ref{1.1}) and (\ref{1.2}). First, we
will derive an energy equality for sufficiently smooth solutions of (\ref{1.1}) and (\ref{1.2}).

\begin{lemma}\label{le:2.1} Assume the conditions (\ref{1.3}), (\ref{1.4}), (\ref{1.5}),
and (\ref{1.6}) hold.  For $0<T\leq+\infty,$ if $(\rho,u,d)\in C^1(Q_T,\mathbb R_+)
\times C^2(Q_T,\mathbb R^3)\times C^2(Q_T,\mathbb S^2)$
is a solution of  (\ref{1.1}) and (\ref{1.2}), then the following energy equality
\beq\label{energy1}
{\bf E}(t)+\int_0^t\int_{\Omega}\big(\mu|\nabla u|^2
+\widetilde{\mu}|\nabla\cdot u|^2+|\triangle d+|\nabla d|^2d|^2\big)
={\bf E}(0),
\eeq
holds for any $0\le t<T$, where ${\bf E}(t)$ and $\mathbf{E}(0)$ are given by (\ref{energyt}) and (\ref{energy0}) respectively.
 \end{lemma}
 \begin{proof} Multiplying ${(\ref{1.1})_2}$ by $u$, integrating the resulting equation
over $\Omega$, applying integration by parts, and using $(\ref{1.1})_1$, we obtain
 \beq\nonumber&&\frac{d}{dt}\int_{\Omega}\left(\frac12\rho|u|^2+\frac{a\rho^{\gamma}}{\gamma-1}\right)+\int_{\Omega}\left(\mu|\nabla u|^2+\widetilde{\mu}|\nabla\cdot u|^2\right)\\
 \label{ener1}&&=-\int_{\Omega}\nabla\cdot(\nabla d\odot\nabla d-\frac12|\nabla d|^2\mathbb I_3) u,
\eeq
where we have used the fact
\begin{eqnarray*}
&&\int_\Omega \rho^\gamma\nabla\cdot u
=\int_\Omega \rho^{\gamma-1}\rho\nabla\cdot u=-\int_\Omega (\partial_t\rho+u\cdot\nabla\rho)\rho^{\gamma-1}
=-\frac{d}{dt}\int_\Omega \frac{\rho^\gamma}{\gamma}+\int_\Omega \frac{\rho^\gamma}{\gamma}
\nabla\cdot u,
\end{eqnarray*}
so that
$$-a\int_\Omega \rho^\gamma\nabla\cdot u=\frac{d}{dt}\int_\Omega \frac{a\rho^\gamma}{\gamma-1}.
$$
Direct calculations show
$$\nabla\cdot(\nabla d\odot\nabla d-\frac12|\nabla d|^2\mathbb{I}_3)=\langle\triangle d,\nabla d\rangle.
$$
Note also, since $|d|=1$, that we have
$$\langle\partial_t d, d\rangle=\langle\nabla d, d\rangle=0,$$
and hence
\begin{equation}\label{2.2.0}
-\int_{\Omega}\nabla\cdot(\nabla d\cdot\nabla d-\frac12|\nabla d|^2\mathbb I_3) u
=-\int_\Omega u\cdot\langle \Delta d+|\nabla d|^2 d, \nabla d\rangle.
\end{equation}
Multiplying $(\ref{1.1})_3$ by $-(\triangle d+|\nabla d|^2d)$ and
integrating over $\Omega$  yields that
\beq\label{ener2}\frac{d}{dt}\int_{\Omega}\frac12\big|\nabla d\big|^2
+\int_{\Omega}\big|\triangle d+|\nabla d|^2d\big|^2
=\int_{\Omega}u\cdot\langle\triangle d+|\nabla d|^2d, \nabla d\rangle.
\eeq
 Putting (\ref{ener1}), (\ref{2.2.0}), and (\ref{ener2}) together
 implies
\beq\frac{d}{dt}{\bf E}(t)+\int_{\Omega}\big(\mu|\nabla u|^2+\widetilde{\mu}|\nabla\cdot u|^2+|\triangle d+|\nabla d|^2d|^2\big)=0.
\eeq
This, after integrating over $t$, implies (\ref{energy1}).
\end{proof}

 In order to construct global finite energy weak solutions to (\ref{1.1})-(\ref{1.2}), we need some important
estimates of transported Ginzburg-Landau equations based on the maximum principle.
 \begin{lemma}\label{le:2.2} For $\epsilon>0$, $T>0,$ and $u_\epsilon\in L^2([0,T], L^\infty(\Omega,\mathbb R^3))$,
assume $d_\epsilon\in L^2([0,T],H^1(\Omega,\mathbb R^3))$, with $(1-|d_\epsilon|^2)\in L^2(Q_T)$,
solves the transported Ginzburg-Landau equation:
 \begin{equation}\label{2.1}
\begin{cases}
\partial_td_\epsilon+u_\epsilon\cdot\nabla d_{\epsilon}=\triangle d_{\epsilon}+\frac{1}{\epsilon^2}(1-|d_\epsilon|^2)d_{\epsilon},& \  {\rm{in}}\ Q_T,\\
\qquad \qquad \ \ \ \ d_\epsilon=g_\epsilon,& \ {\rm{on}}\ \partial_pQ_T.
\end{cases}
\end{equation}
If $g_\epsilon\in H^1(\Omega,\mathbb R^3)$ satisfies $|g_{\epsilon}(x)|\leq 1$ for a.e. $x\in\Omega,$
then
$$|d_{\epsilon}(x,t)|\leq1\ {\rm{for\ a.e.}}\ (x,t)\in Q_T.$$
 \end{lemma}

\begin{proof} We will follow the proof of Lemma 2.1 of Lin-Wang \cite{LW1} with some modifications.
For any $k>1,$ define $f_\epsilon^k: Q_T\to \mathbb R_+$ by
 \beqno f_\epsilon^k=\begin{cases}
k^2-1, &\ \mbox{if}\ |d_\epsilon(x,t)|>k,\\
|d_\epsilon(x,t)|^2-1, & \ \mbox{if}\ 1<|d_\epsilon(x,t)|\leq k, \\
0,  & \ \mbox{if}\ |d_\epsilon(x,t)|\leq1.   \end{cases} \eeqno
By direct calculations, we have that $f_\epsilon^k$ satisfies, in the sense of distributions,
\beq\label{max1}\begin{cases}\partial_tf_\epsilon^k+u_\epsilon\cdot\nabla f_\epsilon^k=\triangle f_\epsilon^k
-2\chi_{\{1<|d_\epsilon|\leq k\}}\Big(|\nabla d_\epsilon|^2+\frac{1}{\epsilon^2}(|d_\epsilon|^2-1)|d_\epsilon|^2\Big)
\leq\triangle f_\epsilon^k \ {\rm{in}}\ Q_T,\\
\ \ \ \ \ \qquad\qquad f_\epsilon^k=0 \ \ \ \ \ \ \ \ \ \ \ \ \ \ \ \ \ \ \ \ \ \ \ \ \ \ \ \
\ \ \ \ \ \ \ \ \ \ \ \ \ \ \ \ \ \ \ \ \ \ \ \ \mbox{on} \ \partial_pQ_T.
\end{cases}
\eeq
Multiplying (\ref{max1}) by $f_\epsilon^k$ and integrating over $\Omega$, we obtain
\beqno\frac{d}{dt}\int_{\Omega}|f_\epsilon^k|^2+2\int_\Omega|\nabla f_\epsilon^k|^2&\leq& 2\int_{\Omega}u_\epsilon\cdot\nabla f_{\epsilon}^kf_\epsilon^k\\
&\leq&\int_{\Omega}|\nabla f_\epsilon^k|^2+\|u_\epsilon(t)\|_{L^\infty(\Omega)}^2\int_{\Omega}|f_\epsilon^k|^2.
\eeqno
Hence we have
\beq\label{max2}\frac{d}{dt}\int_{\Omega}|f_\epsilon^k|^2dx\leq 2\|u_\epsilon(\cdot)\|_{L^\infty(\Omega)}^2\int_{\Omega}|f_\epsilon^k|^2dx.\eeq
Since $u_\epsilon\in L^2([0,T], L^\infty(\Omega))$ and $f_\epsilon^k(x,0)=0$ for a.e.
$x\in\Omega$, applying Gronwall's inequality to (\ref{max2})
yields that
$f_\epsilon^k=0$ a.e. in $Q_T.$
By the definition of $f_\epsilon^k$,  this implies that $d_\epsilon\leq 1$ a.e. in $Q_T.$
\end{proof}

We also have the following lemma.

\begin{lemma}\label{le:2.3} For $\epsilon>0$, $T>0$, and $u_\epsilon\in L^2([0,T],L^\infty(\Omega,\mathbb R^3))$,
assume $d_\epsilon\in L^2([0, T], H^1(\Omega,\mathbb R^3))$, with $(1-|d_\epsilon|^2)\in L^2(Q_T)$,
solves the transported Ginzburg-Landau equation (\ref{2.1}).
If $g_\epsilon\in H^1(\Omega,\mathbb R^3)$ satisfies
$$|g_{\epsilon}(x)|\leq 1\ \mbox{and}\ g_\epsilon^3(x)\geq 0\ {\rm{for\ a.e.}}\  x\in\Omega,
$$
then
$$|d_{\epsilon}(x,t)|\leq1 \ {\rm{and}}\
 d_\epsilon^3(x,t)\geq 0
\ {\rm{for\ a.e.}}\ (x,t)\in Q_T.
$$
\end{lemma}
\begin {proof} We will modify the proof of Lemma 2.2 by Lin-Wang \cite{LW1}. First it follows from
Lemma \ref{le:2.2} that
$$0\leq\frac{1}{\epsilon^2}(1-|d_\epsilon|^2)\leq\frac{1}{\epsilon^2}.$$
Set $\widetilde{d_\epsilon^3}:=e^{-\frac{t}{\epsilon^2}}d_\epsilon^3$. Then we have $$\partial_t\widetilde{d_\epsilon^3}+u_\epsilon\cdot\nabla\widetilde{d_\epsilon^3}-\triangle\widetilde{d_\epsilon^3}
=h_\epsilon\widetilde{d_\epsilon^3},$$
where
$$h_\epsilon(x,t)=\Big(\frac{1}{\epsilon^2}(1-|d_\epsilon|^2)-\frac{1}{\epsilon^2}\Big)\leq 0
\ \ {\rm{a.e.}}\ (x,t)\in Q_T.
$$
Since $\widetilde{d_\epsilon^3}\geq0$ on $\partial_p Q_T$, we have that $\displaystyle (\widetilde{d_\epsilon^3})^{-}:=-\min\big\{\widetilde{d_\epsilon^3},0\big\}$
satisfies
\beq\label{max3}
\begin{cases}\partial_t(\widetilde{d_\epsilon^3})^{-}+u_{\epsilon}\cdot\nabla(\widetilde{d_\epsilon^3})^{-}
-\triangle(\widetilde{d_\epsilon^3})^{-}=h_\epsilon(\widetilde{d_\epsilon^3})^{-},  \ {\rm{in}} \ \ Q_T,\\
\ \ \quad\qquad\qquad\qquad\qquad\qquad (\widetilde{d_\epsilon^3})^{-}=0,  \ \ \ \ \ \ \ \ \ {\rm{on}}\
\partial_p Q_T.
\end{cases}
\eeq
Multiplying $(\ref{max3})_1$ by $(\widetilde{d_\epsilon^3})^{-}$ and integrating the resulting equation
over $\Omega$, we have
\beqno&&\frac{d}{dt}\int_{\Omega}|(\widetilde{d_\epsilon^3})^{-}|^2+2\int_{\Omega}|\nabla(\widetilde{d_\epsilon^3})^{-}|^2\\
&&=-2\int_{\Omega}u_\epsilon\cdot\nabla(\widetilde{d_\epsilon^3})^{-}(\widetilde{d_\epsilon^3})^{-}
+2\int_{\Omega}h_\epsilon|(\widetilde{d_\epsilon^3})^{-}|^2\\
&&\le -2\int_{\Omega}u_\epsilon\cdot\nabla(\widetilde{d_\epsilon^3})^{-}(\widetilde{d_\epsilon^3})^{-}\\
&&\leq\int_{\Omega}|\nabla(\widetilde{d_\epsilon^3})^{-}|^2+\|u_\epsilon(t)\|_{L^\infty(\Omega)}^2\int_{\Omega}|(\widetilde{d_\epsilon^3})^{-}|^2,
\eeqno
where we have used the fact that $h_\epsilon(x,t)\leq 0$ a.e. $(x,t) \in Q_T.$
Thus we have
\beqno\frac{d}{dt}\int_{\Omega}|(\widetilde{d_\epsilon^3})^{-}|^2\leq
\|u_\epsilon(t)\|_{L^\infty(\Omega)}^2\int_{\Omega}|(\widetilde{d_\epsilon^3})^{-}|^2.
\eeqno
Applying Gronwall's inequality and using the initial condition $(\widetilde{d_\epsilon^3})^{-}(x,0)=0$
a.e. $x\in\Omega$, we obtain that $\displaystyle (\widetilde{d_\epsilon^3})^{-}=0$ a.e. in $Q_T.$
Therefore  $\displaystyle d_\epsilon^3\geq0\ \ {\rm{a.e.}}\ \ Q_T.$
This completes the proof of Lemma \ref{le:2.3}.
\end{proof}

\setcounter{section}{2}
\setcounter{equation}{0}
\section{Review of Lin-Wang's compactness results}

In order to show that a family of global finite weak solutions $(\rho_\epsilon, u_\epsilon, d_\epsilon)$
to the Ginzburg-Landau approximation of compressible nematic liquid crystal flow converges
to  a global finite weak solution $(\rho, u, d)$ of the compressible nematic liquid crystal flow (\ref{1.1}) and
(\ref{1.2}), we need to establish the compactness of $d_\epsilon$ in $L^2_{\rm{loc}}([0,T], H^1_{\rm{loc}}(\Omega,\mathbb R^3))$.
Under suitable conditions, this has recently been achieved by Lin-Wang \cite{LW1} in their studies of
the existence of global weak solutions to the incompressible nematic liquid crystal flow.

Since such a compactness property also plays a crucial role in this paper,  we will state it and refer the interested
readers to the paper \cite{LW1} for more detail. For  $a\in (0,2],$ denote
$$\mathbb S^2_{-1+a}=\Big\{y=(y^1,y^2,y^3)\in \mathbb S^2\big|\ y^3\geq-1+a\Big\}.$$
For any $a\in(0,2],$ $L_1>0$ and $L_2>0,$ let
${\bf X}(L_1,L_2, a;\Omega)$ denote the set consisting of all
maps $d_\epsilon\in H^1(\Omega,\mathbb R^3)$, with $\epsilon\in(0,1],$
that are solutions of
\begin{equation}
\label{2.2}\triangle d_\epsilon+\frac{1}{\epsilon^2}(1-|d_\epsilon|^2)d_{\epsilon}=\tau_{\epsilon}
\ \ {\rm{in}} \ \ \Omega, \ {\rm{with}}\  \tau_\epsilon\in L^2(\Omega,\mathbb R^3),
\end{equation}
such that for all $0<\epsilon\le 1$, the following properties hold:
\begin{itemize}
\item [(i)] $|d_\epsilon|\leq 1$ and $d_{\epsilon}^3\geq-1+a$ for a.e. $x\in\Omega$.
\item [(ii)] $\displaystyle
{\bf E}_\epsilon(d_\epsilon):=\int_{\Omega}\big(\frac12|\nabla d_{\epsilon}|^2+\frac{3}{4\epsilon^2}(1-|d_{\epsilon}|^2)^2\big)
\leq L_1$.
\item [(iii)] $\displaystyle\big\|\tau_{\epsilon}\big\|_{L^2(\Omega)}\leq L_2.$
\end{itemize}
We have
\begin{thm}\label{th:2.1}{\rm{(}}\cite{LW1}{\rm{)}} For any $a\in(0,2],$ $L_1>0$, and $L_2>0,$ the set ${\bf X}(L_1,L_2, a;\Omega)$ is precompact in $H^1_{\rm{loc}}(\Omega, \mathbb R^3).$ In particular, if for $\epsilon\rightarrow0,$ $\{d_\epsilon\}\subset H^1(\Omega,\mathbb R^3)$ is a sequence of maps in
${\bf X}(L_1,L_2,a;\Omega),$ then there exists a map $d\in H^1(\Omega,\mathbb S^2_{-1+a})\cap
{\bf Y}(L_1, L_2, a;\Omega)$ such that after passing to possible subsequences, $d_\epsilon\rightarrow d$ in $H^1_{\rm{loc}}(\Omega, \mathbb R^3)$ and
$$e_\epsilon(d_\epsilon)\,dx:=\big(\frac12|\nabla d_\epsilon|^2+\frac{(1-|d_\epsilon|^2)^2}{4\epsilon^2}\big)\,dx
\rightharpoonup \frac12|\nabla d|^2\,dx$$
as convergence of Radon measures.
\end{thm}
The idea of proof of Theorem \ref{th:2.1} is based on: (1) almost energy monotonicity inequality of $d_\epsilon\in {\bf X}(L_1,L_2, a;\Omega)$; (2) an $\delta_0$-regularity and compactness property of $d_\epsilon\in {\bf X}(L_1,L_2, a;\Omega)$;
(3) the blowing-up analysis of $d_\epsilon\in {\bf X}(L_1,L_2, a;\Omega)$ as $\epsilon\rightarrow 0$
in terms of both the concentration set $\Sigma$ and the defect measure $\nu$, motivated by that of
harmonic maps by Lin \cite{Lin} and approximated harmonic maps \cite{LW2, LW3, LW4}; and (4) the ruling out
of possible harmonic $\mathbb S^2$'s generated at $\Sigma$.

In order to study the large time behavior of global finite  energy weak solutions to the compressible nematic liquid crystal flow
(\ref{1.1}) and (\ref{1.2}), we also need the following compactness result on approximated harmonic maps to
$\mathbb S^2_{-1+a}$ for $0<a\le 2$.

For $0<a\le 2$, $L_1>0$, and $L_2>0$, let ${\bf Y}(L_1,L_2, a;\Omega)$ be the set consisting of maps
$d\in H^1(\Omega,\mathbb S^2)$ that are approximated harmonic maps, i.e.,
\begin{equation}\label{approxhm}
\Delta d+|\nabla d|^2 d=\tau \ {\rm{in}}\ \Omega, \ {\rm{with}}\ \tau\in L^2(\Omega, \mathbb R^3),
\end{equation}
that satisfy the following properties:
\begin{itemize}
\item[(i)] $d^3(x)\ge -1+a$ for a.e. $x\in\Omega$.
\item[(ii)] $\displaystyle {\bf F}(d):=\frac12\int_\Omega |\nabla d|^2\le L_1.$
\item[(iii)] $\displaystyle \big\|\tau\big\|_{L^2(\Omega)}\le L_2$.
\item[(iv)] (almost energy monotonicity inequality) for any $x_0\in\Omega$ and $0<r\le R<{\rm{d}}(x_0,\partial\Omega)$,
\begin{equation}\label{almostmono}
\Psi_R(d,x_0)\ge \Psi_r(d,x_0)+\frac12\int_{B_R(x_0)\setminus B_r(x_0)}|x-x_0|^{-1}\big|\frac{\partial d}{\partial |x-x_0|}\big|^2,
\end{equation}
where
$$\Psi_r(d,x_0):=\frac1{r}\int_{B_r(x_0)}\big(\frac12|\nabla d|^2-\langle (x-x_0)\cdot\nabla d, \tau\rangle\big)
+\frac12\int_{B_r(x_0)}|x-x_0||\tau|^2.
$$
\end{itemize}

\begin{thm}\label{th:2.2}{\rm{(}}\cite{LW1}{\rm{)}}
For any $a\in(0,2],$ $L_1>0$, and $L_2>0,$ the set ${\bf Y}(L_1,L_2, a;\Omega)$ is precompact in $H^1_{\rm{loc}}(\Omega, \mathbb S^2).$ In particular, if $\{d_i\}\subset H^1(\Omega,\mathbb R^3)$ is a sequence of approximated harmonic maps in ${\bf Y}(L_1,L_2,a;\Omega)$ with tension fields $\{\tau_i\}$, then there exist
$\tau_0\in L^2(\Omega,\mathbb R^3)$ and an approximated harmonic map
$d_0\in {\bf Y}(L_1,L_2, a; \Omega)$ with tension field $\tau_0$ such that after passing to possible subsequences, $d_i\rightarrow d_0$ in $H^1_{\rm{loc}}(\Omega, \mathbb S^2)$ and $\tau_i\rightharpoonup \tau_0$
in $L^2(\Omega,\mathbb R^3)$. In fact,  $\{d_i\}$ is bounded in $H^2_{\rm{loc}}(\Omega,\mathbb S^2)$.
In particular, $d_0\in H^{2}_{\rm{loc}}(\Omega,\mathbb S^2)$.
\end{thm}

\setcounter{section}{3}
\setcounter{equation}{0}
\section{Ginzburg-Landau approximation of compressible
nematic liquid crystal flow}

In this section, we will consider the Ginzburg-Landau approximation of compressible nematic liquid crystal flow
and state the existence of global weak solutions, which is an improved version of an earlier result obtained
by Wang-Yu \cite{WY1} (see also \cite{LQ}).

For $\epsilon>0$ and $0<T\le +\infty$, the Ginzburg-Landau approximation equation
of (\ref{1.1}) and (\ref{1.2}) seeks $(\rho_\epsilon, u_\epsilon, d_\epsilon):
Q_T\to\mathbb R_+\times \mathbb R^3\times \mathbb R^3$ that satisfies:
\begin{equation}
\label{3.3}
\begin{cases}
\partial_t\rho_\epsilon+\nabla\cdot(\rho u_\epsilon)=0,\\
\partial_t(\rho_\epsilon u_\epsilon)+\nabla\cdot(\rho_\epsilon u_\epsilon\otimes u_\epsilon)+a\nabla\rho^{\gamma}_\epsilon\\
=\mathcal{L}u_\epsilon-\nabla\cdot\big(\nabla d_\epsilon\odot\nabla d_\epsilon-(\frac12|\nabla d_\epsilon|^2+\frac{1}{4\epsilon^2}(1-|d_\epsilon|^2)^2)\mathbb{I}_3\big),\\
\partial_t d_\epsilon+u_\epsilon\cdot\nabla d_\epsilon=\triangle d_\epsilon+\frac{1}{\epsilon^2}(1-|d_\epsilon|^2)d_\epsilon,
\end{cases}
\end{equation}
along with the initial and boundary condition (\ref{1.2}). We would like to point out
that the notion of finite energy weak solutions of (\ref{3.3}) and (\ref{1.2}) can be
defined in the same way as that of (\ref{1.1}) and (\ref{1.2}) given in \S 1.

\begin{thm}\label{th:3.2} Assume $\gamma>\frac32$ and  the condition (\ref{1.3}), and  $(\rho_0,m_0,d_0)$
satisfies (\ref{1.4}), (\ref{1.5}), (\ref{1.6}). Then there exists a global
finite energy weak solution $(\rho_\epsilon,u_{\epsilon},d_{\epsilon}):\Omega\times [0,+\infty)\to\mathbb R_+\times \mathbb R^3\times
\mathbb R^3$ to the system (\ref{3.3}), under the initial and boundary condition (\ref{1.2}),
such that
\begin{itemize}
\item[(i)] $d_{\epsilon}=(d^1_{\epsilon},d^2_{\epsilon},d^3_{\epsilon})\in L^{\infty}((0,\infty), H^1(\Omega, \mathbb R^3))$, with $|d_\epsilon|\leq 1$ and $d^3_{\epsilon}\geq0$ for a.e. $(x,t)\in \Omega\times (0,\infty)$.
\item[(ii)] $(\rho_\epsilon, u_\epsilon, d_\epsilon)$ satisfies the global energy inequality
\begin{eqnarray}
\frac{d}{dt}{\bf F}_{\epsilon}(t)+\int_{\Omega}\big(\mu|\nabla u_{\epsilon}|^2
+\widetilde{\mu}|\nabla\cdot u_{\epsilon}|^2+|\triangle d_{\epsilon}
+\frac{1}{\epsilon^2}(1-|d_\epsilon|^2)d_{\epsilon}|^2\big)(t)
\label{3.6}\leq 0
\end{eqnarray}
in $\mathcal D'((0,+\infty))$,
where
$${\bf F}_{\epsilon}(t):=\int_{\Omega}\Big(\frac12\rho_{\epsilon}|u_\epsilon|^2
+\frac{a\rho^{\gamma}_{\epsilon}}{\gamma-1}
+\big(\frac12|\nabla d_{\epsilon}|^2+\frac{1}{4\epsilon^2}(1-|d_{\epsilon}|^2)^2\big)\Big)(t).$$
\end{itemize}
\end{thm}

\begin{proof} The existence of finite energy weak solutions has been established by Wang-Yu \cite{WY1}, which
uses a three level approximation scheme similar to that of compressible Navier-Stokes equation by
\cite{F1} and \cite{FNP1}. It consists of Faedo-Galerkin approximation, artificial viscosity, and artificial
pressure. The reader can consult the proof of \cite{WY1} Theorem 2.1 for the detail.

Here we only indicate the proof of (i).
Let $\epsilon>0$ be fixed. Recall that  the first level of Faedo-Galerkin's approximation involves
to solve the initial and boundary value problems of (\ref{3.3}) as follows. For any $\alpha>0$, $\delta>0$,  and $0<T<+\infty$,
we first approximate the initial data $(\rho_0, m_0,d_0)$ by $\big(\rho_{0,\delta}, m_{0,\delta}, d_{0,\delta}\big)
\in C^2(\overline\Omega,\mathbb R_+\times \mathbb R^3\times \mathbb R^3_+)$ such that
the following conditions hold:
\begin{equation}\label{approx_initial}
\begin{cases}\delta\le\rho_{0,\delta}\le \delta^{-1} \ {\rm{in}}\ \Omega, \frac{\partial\rho_{0,\delta}}{\partial\nu}\big|_{\partial\Omega}=0,
\ {\rm{and}}\  \rho_{0,\delta}\rightarrow\rho_0 \ {\rm{in}}\ L^\gamma(\Omega),\\
\ m_{0,\delta}\rightarrow m_0
\ {\rm{in}}\ L^{\frac{2\gamma}{\gamma+1}}(\Omega), \frac{|m_{0,\delta}|^2}{\rho_{0,\delta}}\rightarrow
\frac{|m_0|^2}{\rho_0}\chi_{\{\rho_0>0\}}\ {\rm{in}}\ L^1(\Omega),\\
|d_{0,\delta}(x)|\le 1, \ d_{0,\delta}^3(x)\ge 0 \ {\rm{a.e.}}\ x\in\Omega,
\ d_{0,\delta}\rightarrow d_0 \ {\rm{in}}\ H^1(\Omega,\mathbb R^3),
\end{cases}
\end{equation}
as $\delta\rightarrow 0$.

For $u\in C^1([0,T], C^2_0(\overline\Omega,\mathbb R^3))$, with
$\displaystyle u\big|_{t=0}=u_{0,\delta}\equiv\frac{m_{0,\delta}}{\rho_{0,\delta}}$,   let
$d_{\delta}=d_{\delta}([u])\in C^1([0,T], C^2(\overline\Omega,\mathbb R^3))$
be the unique solution of (see \cite{WY1} Lemma 3.1 and Lemma 3.2):
\begin{equation}\label{d-eqn}
\begin{cases}
\partial_t d+u\cdot\nabla d=\Delta d+\frac{1}{\epsilon^2}(1-|d|^2)d & \ {\rm{in}}\ Q_T,\\
\qquad\qquad \ \ d=d_{0,\delta}   & \ {\rm{on}}\ \partial_p Q_T.
\end{cases}
\end{equation}
Since $|d_{0,\delta}(x)|\le 1$ and $d^3_{0,\delta}(x)\ge 0$
for $x\in\Omega$, it
follows from Lemma \ref{le:2.2} and Lemma \ref{le:2.3} that $d_\delta$ satisfies
\begin{equation}\label{image}
|d_{\delta}(x,t)|\le 1 \ {\rm{and}}\ d_{\delta}^3(x,t)\ge 0,
\ \forall\ (x,t)\in Q_T.
\end{equation}
Now let $\rho_{\alpha, \delta}=\rho_{\alpha,\delta}([u])\in C^1([0,T], C^2(\overline\Omega))$
be the unique solution of the problem:
\begin{equation}\label{rho-eqn}
\begin{cases}
\partial_t\rho+\nabla\cdot(\rho u)=\alpha\Delta\rho & \ {\rm{in}}\  Q_T,\\
\rho(x,0)=\rho_{0,\delta}(x)  & \ {\rm{in}}\ \Omega,\\
\qquad\frac{\partial\rho}{\partial\nu}= 0 & \ {\rm{on}}\ \partial\Omega\times (0,T).
\end{cases}
\end{equation}
While for $u$, it involves to employ first the Galerkin method and then the fixed point theorem
to solve $u=u_{\alpha,\delta}([u])$ to the problem: for some $\beta>\max\{4,\gamma\}$,
\begin{equation}\label{u-eqn}
\begin{cases}
\partial_t(\rho_{\alpha,\delta}u)+\nabla\cdot(\rho_{\alpha,\delta} u\otimes u)
+a\nabla \big(\rho_{\alpha,\delta}^\gamma\big)
+\delta\nabla \big(\rho_{\alpha,\delta}^\beta\big)
+\alpha\nabla u\cdot\nabla\rho_{\alpha,\delta}\\
=\mathcal Lu -\nabla\cdot\Big[\nabla d_{\delta}\odot \nabla d_{\delta}
-\big(\frac12|\nabla d_{\delta}|^2+\frac{1}{4\epsilon^2}(1-|d_{\delta}|^2)^2\big)
\mathbb I_3\Big], \ {\rm{in}}\ Q_T,\\
u=u_{0,\delta}\ \ {\rm{on}}\ \partial_p Q_T.
\end{cases}
\end{equation}
Since the global weak solution $(\rho_\epsilon, u_\epsilon, d_\epsilon)$ to the system (\ref{3.3}), under the initial and boundary condition
 (\ref{1.2}), constructed in \cite{WY1}, was obtained as a strong limit of $(\rho_{\alpha,\delta}, u_{\alpha,\delta}, d_{\delta})$ in
$L^\gamma(Q_T)\times L^2(Q_T)\times L^2([0,T], H^1(\Omega,\mathbb R^3))$
for any $0<T<+\infty$, as viscosity coefficients $\alpha\rightarrow 0$ first and then artificial pressure
coefficients $\delta\rightarrow 0$. It is readily seen
that $d_\epsilon$ satisfies the property (ii).
\end{proof}

\setcounter{section}{4}
\setcounter{equation}{0}
\section{Existence of global weak solutions}

In this section, we will prove Theorem \ref{th:1.1} by studying in depth the convergence of sequences of solutions $(\rho_\epsilon,u_\epsilon,d_{\epsilon})$, constructed by Theorem \ref{th:3.2}, as $\epsilon\rightarrow 0^+.$

\bigskip
\noindent{\bf Proof of Theorem \ref{th:1.1}}.

\medskip
To prove the  existence of global finite energy weak solutions to (\ref{1.1}), let $(\rho_\epsilon,u_{\epsilon},d_{\epsilon}):\Omega\times [0,+\infty)
\to\mathbb R_+\times \mathbb R^3\times \mathbb R^3$, $0<\epsilon\le 1$, be a family of finite energy weak solutions to
 the system (\ref{3.3}), under the initial and boundary condition (\ref{1.2}), constructed by Theorem \ref{th:3.2}.
Since $|d_0|=1$ and $d_0^3\geq0$ a.e. in $\Omega$, $(\rho_\epsilon,u_{\epsilon},d_{\epsilon})$
satisfies all these properties in Theorem \ref{th:3.2}. In particular, it follows from (\ref{3.6}) that
\beq\nonumber&&\sup\limits_{\epsilon>0}\Big[\sup\limits_{0<t<\infty}\int_{\Omega}\Big(
\frac12\rho_\epsilon|u_\epsilon|^2+\frac{a}{\gamma-1}\rho^{\gamma}_{\epsilon}
+(\frac12|\nabla d_{\epsilon}|^2+\frac{1}{4\epsilon^2}(1-|d_{\epsilon}|^2)^2)\Big)(t)\\
&&\nonumber+\int_{0}^{\infty}\int_{\Omega}\Big(\mu|\nabla u_{\epsilon}|^2+\widetilde{\mu}|\nabla\cdot u_{\epsilon}|^2+|\triangle d_{\epsilon}+\frac{1}{\epsilon^2}(1-|d_\epsilon|^2)d_{\epsilon}|^2\Big)\Big]\\
&&\label{global_bd}
\leq \int_{\Omega}\Big(\frac{|m_0|^2}{2\rho_{0}}\chi_{\{\rho_{0}>0\}}+\frac{a}{\gamma-1}\rho^{\gamma}_{0}
+\frac12|\nabla d_{0}|^2\Big):={\bf E}(0).
\eeq
By (\ref{global_bd}), we may assume that there exists $(\rho, u, d):\Omega\times [0,+\infty)
\to \mathbb R_+\times \mathbb R^3\times \mathbb S^2$ such that after passing to a subsequence,
\begin{equation}\label{weak_conv}
\begin{cases}
\rho_\epsilon\rightharpoonup \rho \ {\rm{weak^*\ in}}\ L^\infty([0,T], L^\gamma(\Omega)),\\
u_\epsilon\rightharpoonup u \ {\rm{in}}\ L^2([0,T], H^1_0(\Omega)),\\
d_\epsilon\rightharpoonup d \ {\rm{weak^*\ in}}\ L^\infty([0,T], H^1(\Omega)),
\end{cases}
\end{equation}
as $\epsilon\rightarrow 0$, for any $0<T<+\infty$.

We will prove that $(\rho, u, d)$ is a global finite energy weak solution to (\ref{1.1}) and (\ref{1.2}).
The proof will be divided into several subsections.

\subsection{$d_\epsilon\rightarrow d$\ strongly in $L^2([0,T], H^1_{\rm{loc}}(\Omega))$}

This will be achieved by applying Theorem \ref{th:2.1}, similar to that of \cite{LW1}.
First it follows from the equation $(\ref{3.3})_3$ and the inequality (\ref{global_bd}) that
$\partial_t d_\epsilon\in L^2([0,T], L^\frac32(\Omega))+L^2([0,T], L^2(\Omega))$ so that
$\partial_t d_\epsilon\in L^2([0,T], H^{-1}(\Omega))$ and
\begin{equation}
\label{4.3}\sup\limits_{0<\epsilon\le 1}\big\|\partial_td_\epsilon\big\|_{L^2(0,T;H^{-1}(\Omega))}<+\infty.
\end{equation}
By Aubin-Lions' lemma, we conclude that
\begin{equation}
\label{4.4}d_\epsilon\rightarrow d\ \ \mbox{in}\ L^2(Q_T)
\ {\rm{and}}\ \nabla d_{\epsilon}\rightharpoonup \nabla d\ \ \mbox{in}\ L^2([0,T], L^2(\Omega)).
\end{equation}
By Fatou's lemma,  (\ref{global_bd}) implies that
\begin{equation}
\label{4.5}\int_0^T\liminf\limits_{\epsilon\rightarrow0}\int_{\Omega}\big|\triangle d_\epsilon+\frac{1}{\epsilon^2}(1-|d_\epsilon|^2)d_{\epsilon}\big|^2\leq {\bf E}(0).
\end{equation}
For sufficiently large $\Lambda>1,$
define the set of good time slice, $G_\Lambda^T$, by
$$G_\Lambda^T:=\Big\{t\in[0,T]\ \Big|\ \liminf\limits_{\epsilon\rightarrow0}\int_{\Omega}\big|\triangle d_\epsilon+\frac{1}{\epsilon^2}(1-|d_\epsilon|^2)d_{\epsilon}\big|^2(t)\leq\Lambda\Big\},$$
and the set of bad time slices, $B_{\Lambda}^T$, by
$$B_\Lambda^T:=[0,T]\backslash G_\Lambda^T
=\Big\{t\in[0,T]\ \Big|\ \liminf\limits_{\epsilon\rightarrow0}\int_{\Omega}
\big|\triangle d_\epsilon+\frac{1}{\epsilon^2}(1-|d_\epsilon|^2)d_{\epsilon}\big|^2(t)>\Lambda\Big\}.$$
 It is easy to see from (\ref{4.5}) that
\begin{equation}
\label{est_badset}
\Big|B_\Lambda^T\Big|\leq \frac{{\bf E}(0)}{\Lambda}.
\end{equation}
By (\ref{global_bd}) and (\ref{est_badset}),  we obtain
 \beq\label{4.7}
\int_{B_\Lambda^T}\int_{\Omega}\Big[|\nabla d_\epsilon-\nabla d|^2+
 \frac{1}{\epsilon^2}(1-|d_\epsilon|^2)^2(t)\Big]
 \leq C \big|B_{\Lambda}^T\big|\sup\limits_{0<t<T}\mathbf F_{\epsilon}(t)
\leq \frac{C{\bf E}(0)}{\Lambda}.
\eeq

For any $t\in G_{\Lambda}^T$, set
$\displaystyle\tau_\epsilon(t)=\big(\triangle d_\epsilon+\frac{1}{\epsilon^2}(1-|d_\epsilon|^2)d_{\epsilon}\big)(t)$. Then it follows from the definition of $G_\Lambda^T$ that
there exists $\tau(t)\in L^2(\Omega,\mathbb R^3)$ such that, after passing to a subsequence,
$\tau_\epsilon(t)\rightharpoonup \tau(t)$ in $L^2(\Omega)$.
Since $\big\{d_\epsilon(t)\big\}\subset {\bf X}({\bf E}(0),\Lambda, 1;\Omega)$,
Theorem \ref{th:2.1} implies that there exists $d(t)\in {\bf Y}({\bf E}(0),\Lambda, 1;\Omega)$
such that after passing to a subsequence, $d_{\epsilon}(t)\rightarrow d(t)$ strongly in $H^1_{\rm{loc}}({\Omega})$ and $\frac{1}{\epsilon^2}(1-|d_\epsilon(t)|^2)^2\rightarrow 0$
in $L^1_{\rm{loc}}(\Omega)$.

Now we want to show that, after passing to a subsequence,
\begin{equation}\label{good-conv}
\nabla d_\epsilon\rightarrow \nabla d  \ \ {\rm{in}}\ \ L^2_{\rm{loc}}(\Omega\times G_\Lambda^T).
\end{equation}
This can be done similarly to Claim 8.2 of \cite{LW1}. Here we provide it.
Suppose (\ref{good-conv}) were false.
Then there exist a subdomain $\widetilde\Omega\subset\subset\Omega$, $\delta_0>0$, and $\epsilon_i\rightarrow 0$ such that
\begin{equation}\label{gap1}\int_{\widetilde\Omega\times G_\Lambda^T}
|\nabla(d_{\epsilon_i}-d)|^2\ge \delta_0.
\end{equation}
Note that from (\ref{4.4}) we have
\begin{equation}
\label{no-gap}
\lim_{\epsilon_i\rightarrow 0}\int_{\Omega\times G_\Lambda^T}
|d_{\epsilon_i}-d|^2=0.
\end{equation}
By Fubini's theorem, (\ref{gap1}), and (\ref{no-gap}), we have that there exists
$t_i\in G_\Lambda^T$ such that
\begin{equation}\label{no-gap1}
\lim_{\epsilon_i\rightarrow 0}\int_{\Omega}|d_{\epsilon_i}(t_i)
-d(t_i)|^2=0,
\end{equation}
and
\begin{equation}\label{gap2}
\int_{\widetilde\Omega}\big|\nabla(d_{\epsilon_i}(t_i)
-d(t_i))\big|^2\ge\frac{2\delta_0}{T}.
\end{equation}
It is easy to see that $\big\{d_{\epsilon_i}(t_i)\big\}\subset {\bf X}({\bf E}(0), \Lambda, 1; \Omega)$
and $\big\{d(t_i)\big\}\subset {\bf Y}({\bf E}(0), \Lambda, 1; \Omega)$. It follows
from Theorem 3.1 and Theorem 3.2 that there exist $d_1, d_2\in {\bf Y}({\bf E}(0), \Lambda, 1; \Omega)$ such that
$$d_{\epsilon_i}(t_i)\rightarrow d_1
\ {\rm{and}}\ d(t_i)\rightarrow d_2\ \ {\rm{in}}\ \ L^2(\Omega)\cap H^1(\widetilde\Omega).
$$
This and (\ref{gap2}) imply that
\begin{equation}\label{gap3}
\int_{\widetilde\Omega}\big|\nabla(d_1
-d_2)\big|^2\ge\frac{2\delta_0}{T}.
\end{equation}
On the other hand, from (\ref{no-gap1}), we have that
\begin{equation}\label{no-gap2}
\int_{\Omega}|d_1
-d_2|^2=0.
\end{equation}
It is clear that (\ref{gap3}) contradicts (\ref{no-gap2}). Hence (\ref{good-conv}) is proven.
Similar to Lemma 4.1, Claim 4.4 in \cite{LW1}, we also have
\begin{equation}
\label{4.10}
\int_{\widetilde\Omega\times G_\Lambda^T}\frac{1}{\epsilon^2}(1-|d_{\epsilon}|^2)^2
\rightarrow 0\ \mbox{as}\ \epsilon\rightarrow 0.
\end{equation}
Combining (\ref{4.7}), (\ref{good-conv}), with (\ref{4.10}), we obtain
\begin{equation}\label{4.11}
\lim_{\epsilon\rightarrow 0}\Big[\|d_\epsilon-d\|_{L^2([0,T],H^1(\widetilde\Omega))}
+\int_{\widetilde\Omega\times [0,T]}\frac{(1-|d_\epsilon|^2)^2}{\epsilon^2}\Big]\le C\Lambda^{-1}.
\end{equation}
Since $\Lambda>1$ can be chosen arbitrarily large, we conclude that
\beq\label{4.12}
d_\epsilon\rightarrow d\ \mbox{in}\ L^2([0,T],H^{1}_{\rm{loc}}(\Omega))
\ {\rm{and}}\  \frac{(1-|d_\epsilon|^2)^2}{\epsilon^2}\rightarrow 0 \ \mbox{in}\ L^1([0,T],
L^1_{\rm{loc}}(\Omega)).
\eeq

\subsection{$\rho_{\epsilon}u_{\epsilon}\rightarrow\rho u$ in the sense of distributions}

By (\ref{global_bd}),  $\sqrt{\rho_\epsilon}$ is bounded in $L^{\infty}([0,T], L^{2\gamma}(\Omega))$
and $\sqrt{\rho_{\epsilon}}u_{\epsilon}$ is bounded in $L^\infty([0,T], L^2(\Omega))$. Thus
$\rho_{\epsilon}u_{\epsilon}$ is bounded in $L^\infty([0,T], L^{\frac{2\gamma}{\gamma+1}}(\Omega))$ and
$\partial_t\rho_{\epsilon}=-\nabla\cdot(\rho_{\epsilon}u_{\epsilon})$ is bounded
in $L^{\infty}([0,T], W^{-1,\frac{2\gamma}{\gamma+1}}(\Omega)).$
Applying \cite{L1} Lemma C.1, we have
\beq\label{4.13}
\rho_\epsilon\rightarrow\rho\ \mbox{in}\ C([0,T],L^{\gamma}_{\rm{weak}}(\Omega)).
\eeq
Since $L^{\gamma}(\Omega)\subset H^{-1}(\Omega)$ is compact,
 we conclude that
\beq\label{4.14}
\rho_\epsilon\rightarrow\rho\ \mbox{in}\ C([0,T],H^{-1}(\Omega)).
\eeq
Thus we show that
\beq\label{4.15}
\rho_{\epsilon}u_{\epsilon}\rightarrow\rho u\ \mbox{in}\ \mathcal D'(Q_T).
\eeq

\subsection{Higher integrability estimates of $\rho_\epsilon$}

{There exist $\theta>0$ and $C>0$ depending only on $\gamma$ and $T$
such that for any $0<\epsilon\le 1$, it holds
\begin{equation}\label{rho-bd}
\int_0^T\int_{\Omega}\rho_\epsilon^{\gamma+\theta}
\leq C.
\end{equation}}

 By Theorem \ref{th:3.2},  $(\rho_{\epsilon},u_{\epsilon})$ is a renormalized solution of (\ref{3.3})$_1$.
Let  $(\rho_{\epsilon},u_{\epsilon}):\mathbb R^3\times (0,T)\to\mathbb R_+\times\mathbb R^3$
be the extension of $(\rho_\epsilon, u_\epsilon)$ from $\Omega$ such that
$(\rho_\epsilon, u_\epsilon)=(0,0)$ in $\mathbb R^3\setminus\Omega$.
Then $(\rho_\epsilon, u_\epsilon)$ satisfies, in the sense of distributions, that
\beq\label{4.17}\partial_t(b(\rho_\epsilon))+\nabla\cdot(b(\rho_{\epsilon})u_{\epsilon})+\big(b'(\rho_{\epsilon})\rho_{\epsilon}-
 b(\rho_{\epsilon})\big)\nabla\cdot u_{\epsilon}=0\ \mbox{in}\ \mathbb R^3\times (0,T),
\eeq
for any bounded function $b\in C^1((0,+\infty))\cap C([0,+\infty))$ (see, e.g., \cite{FNP1}).

As in \cite{F1}, \cite{FNP1} and \cite{WY1}, we can employ suitable approximations
so that (\ref{4.17}) also holds for $b(\rho_{\epsilon})=\rho_{\epsilon}^{\theta}$ for $0<\theta<1$.
Note that $\rho_\epsilon^\theta\in L^{\frac{\gamma}{\theta}}(Q_T).$
For $m\ge 1$, let $S_m(f)=\eta_{\frac{1}{m}}*f$ denote the standard mollification of $f\in L^1(\mathbb R^3)$.
Then we have
 \beq
\label{4.18}
\partial_t\big(S_m(\rho_\epsilon^\theta)\big)+\nabla\cdot\big(S_m(\rho_{\epsilon}^\theta)u_{\epsilon}\big)
-(1-\theta)S_m\big(\rho_{\epsilon}^\theta\nabla\cdot u_{\epsilon}\big)=
q_m\ \ \mbox{in}\ \mathbb R^3\times (0,T),
\eeq
where
$$q_m=\nabla\cdot(S_m(\rho_{\epsilon}^\theta)u_{\epsilon})-S_m(\nabla\cdot(\rho_\epsilon^\theta u_{\epsilon})).$$
By virtue of \cite{L1} Lemma 2.3, $\rho_{\epsilon}^\theta\in L^{\infty}([0,T], L^{\frac{\gamma}{\theta}}(\Omega))$,
and $u_{\epsilon}\in L^2([0,T], H^{1}_0(\Omega))$,  we have that
\begin{equation}
\label{error-est}
\lim\limits_{m\rightarrow\infty}\big\|q_m\big\|_{L^2([0,T], L^\lambda(\mathbb R^3))}=0,
\ {\rm{with}}\ \frac{1}{\lambda}=\frac{\theta}{\gamma}+\frac12,
\end{equation}
provided $\theta<\frac{\gamma}{2}.$

As in \cite{F1} and \cite{L1},   define the (inverse of divergence) operator
$$\mathcal B: \Big\{f\in L^p(\Omega) \ \big |\ \int_{\Omega}f=0\Big\}\mapsto W^{1,p}_0(\Omega,\mathbb R^3)$$
such that for any $1<p<+\infty$,
\begin{equation}\label{4.19}
\begin{cases}
\nabla\cdot \mathcal B(f)=f\  \mbox{in}\ \Omega, \ \ \mathcal B(f)=0 \ {\rm{on}}\ \partial\Omega,\\
\big\|\mathcal B(f)\big\|_{W^{1,p}_{0}(\Omega)}\leq C(p)\big\|f\big\|_{L^p(\Omega)}.
\end{cases}
\end{equation}
Set $\displaystyle\oint_{\Omega} f=\frac{1}{|\Omega|}\int_{\Omega}f.$ For
$\varphi\in C_0^{\infty}((0,T))$, with $0\leq\varphi\leq1$, let
$$\phi(x,t)=\varphi(t)\mathcal B\Big[S_m(\rho_{\epsilon}^{\theta})-\oint_{\Omega}S_m(\rho_{\epsilon}^{\theta})\Big](x,t).$$
  By (\ref{global_bd}), we see that for sufficiently small $\theta$,  $$\Big[S_m(\rho_{\epsilon}^{\theta})-\oint_{\Omega}S_m(\rho_{\epsilon}^{\theta})\Big]
\in C([0,T], L^p(\Omega)),\ \forall\ p\in(1,+\infty).$$
  By (\ref{4.19}) and the Sobolev embedding theorem, we have
that $\phi\in C(Q_T)$. Thus we can test the equation
$(\ref{3.3})_2$ by $\phi$ and obtain
  \beqno&&a\int_0^T\int_{\Omega}\varphi(t)\rho_{\epsilon}^\gamma S_m(\rho_\epsilon^{\theta})\\
&&=a\int_0^T\varphi(t)\big(\int_{\Omega}\rho_{\epsilon}^\gamma \big)\big(\oint_{\Omega}S_m(\rho_{\epsilon}^{\theta})\big)\\
  &&\ -\int_0^T\int_{\Omega}\varphi'(t)\rho_{\epsilon} u_{\epsilon}\mathcal B\big[S_m(\rho_\epsilon^\theta)-\oint_{\Omega}S_m(\rho_{\epsilon}^{\theta})\big]\\
   &&\ +\int_0^T\int_{\Omega}\varphi(t)(\mu\nabla u_{\epsilon}-\rho_\epsilon u_{\epsilon}\otimes u_{\epsilon})\nabla \mathcal B\big[S_m(\rho_\epsilon^\theta)-\oint_{\Omega}S_m(\rho_{\epsilon}^{\theta})\big]\\
  &&\ +\int_0^T\int_{\Omega}\varphi(t)\widetilde{\mu}\nabla\cdot u_{\epsilon}\nabla\cdot \mathcal B\big[S_m(\rho_\epsilon^\theta)-\oint_{\Omega}S_m(\rho_{\epsilon}^{\theta})\big]\\
  &&\ +(1-\theta)\int_0^T\int_{\Omega}\varphi(t)\rho_{\epsilon} u_{\epsilon}\mathcal B\big[S_m(\rho_\epsilon^\theta\nabla\cdot u_{\epsilon})-\oint_{\Omega}S_m(\rho_\epsilon^\theta\nabla\cdot u_{\epsilon})\big]\\
  &&\ +\int_0^T\int_{\Omega}\varphi(t)\big(\triangle d_{\epsilon}+\frac{1}{\epsilon^2}(1-|d_{\epsilon}|^2)d_{\epsilon}\big)\cdot\nabla d_{\epsilon}\mathcal B\big[S_m(\rho_{\epsilon}^{\theta})-\oint_{\Omega}S_m(\rho_{\epsilon}^\theta)\big]\\
  &&\ -\int_0^T\int_{\Omega}\varphi(t)\rho_{\epsilon} u_{\epsilon}\mathcal B\big[\nabla\cdot (S_m(\rho_\epsilon^\theta) u_{\epsilon})\big]\\
&&\ +\int_0^T\int_{\Omega}\varphi(t)\rho_{\epsilon} u_{\epsilon}\mathcal B\big[q_m-\oint_{\Omega}q_m\big]\\
&&=\sum\limits_{i=1}^{7}L_i^m+\int_0^T\int_{\Omega}\varphi(t)\rho_{\epsilon} u_{\epsilon}\mathcal B\big[q_m-\oint_{\Omega}q_m\big].
\eeqno
Since $\rho_\epsilon u_\epsilon$ is bounded in $L^\infty([0,T], L^{\frac{2\gamma}{\gamma+1}}(\Omega))$ and
$q_m$ satisfies (\ref{error-est}), it follows from (\ref{4.19}), the Sobolev embedding theorem,
and the H\"older inequality that
$$\lim_{m\rightarrow\infty}\int_0^T\int_{\Omega}\varphi(t)\rho_{\epsilon} u_{\epsilon}\mathcal B\big[q_m-\oint_{\Omega}q_m\big]=0.$$
Hence, after taking $m\rightarrow\infty$, we have
\begin{equation}
\label{improve1}\int_{(0,T)\times\Omega}\varphi\rho_{\epsilon}^{\gamma+\theta}\leq
\limsup_{m\rightarrow\infty}\sum\limits_{i=1}^{7}L_i^m.
\end{equation}
Now we estimate $L_1^m,\cdots, L^m_7$ as follows.

\noindent(1) For $L_1^m,$ we have that
\beqno
&&\big|\lim_{m\rightarrow\infty}L_1^m\big|
=\Big|a\int_0^T\int_{\Omega}\varphi\rho_\epsilon^\gamma(\oint_{\Omega}\rho_\epsilon^\theta)\Big|
\leq C\Big(\int_0^T\int_{\Omega}\rho_\epsilon^\gamma\Big)\Big\|\rho_\epsilon\Big\|_{L^\infty([0,T], L^\gamma(\Omega))}^\theta
\eeqno
is uniformly bounded.

\noindent(2) For $L_2^m$, we have that
\beqno
|L_2^m|&=&\Big|\int_0^T\int_{\Omega}\varphi'(t)\rho_{\epsilon}u_{\epsilon}\mathcal
B\big[q_m-\oint_{\Omega}q_m\big]\Big|\\
&\leq& C\int_0^T\big\|\rho_\epsilon u_\epsilon\big\|_{L^{\frac{2\gamma}{\gamma+1}}(\Omega)}
\big\|\mathcal B\big[q_m-\oint_{\Omega}q_m\big]\big\|_{L^{\frac{2\gamma}{\gamma-1}}(\Omega)}dt\\
&\leq& C\int_0^T\big\|\mathcal B\big[q_m-\oint_{\Omega}q_m\big]\big\|_{W^{1,\lambda}(\Omega)}dt\\
&\leq& C\int_0^T\big\|q_m-\oint_{\Omega}q_m\big\|_{L^{\lambda}(\Omega)}dt\\
&\leq& C\Big\|q_m\Big\|_{L^2([0,T],L^{\lambda}(\Omega))}\rightarrow 0
\ \ \mbox{as}\ \ m\rightarrow+\infty,
\eeqno
provided $\theta<\frac{\gamma}{3}-\frac12$.

\noindent (3) For $L_3^m$, we have
\beqno
\big|L_3^m\big|
&\leq& C\int_0^T\Big\{\big\|u_\epsilon\big\|_{H^1(\Omega)}
\big\|\mathcal B\big[\rho_\epsilon^\theta-\oint\rho_\epsilon^\theta\big]\big\|_{H^1(\Omega)}\\
&&+\big\|\rho_\epsilon\big\|_{L^\gamma(\Omega)}\big\|u_\epsilon\big\|_{L^6(\Omega)}^2
\big\|\mathcal B\big[\rho_\epsilon^\theta-\oint\rho_\epsilon^\theta \big]\big\|_{W^{1,\frac{3\gamma}{2\gamma-3}}(\Omega)}\Big\}dt\\
&\leq& C\int_0^T\Big(\big\|u_\epsilon\big\|_{H^1(\Omega)}\big\|\rho_\epsilon\big\|_{L^{2\theta}(\Omega)}^\theta
+\big\|\rho_\epsilon\big\|_{L^\gamma(\Omega)}\big\|u_\epsilon\big\|_{H^1(\Omega)}^2
\big\|\rho_\epsilon\big\|_{L^{\frac{3\gamma\theta}{2\gamma-3}}(\Omega)}^\theta\Big)dt
\eeqno
is uniformly bounded, provided $\theta<\min\big\{\frac{\gamma}{2},\frac{2\gamma}{3}-1\big\}.$

\noindent (4) For $L_4^m$, we have
\beqno
\big|L_4^m\big|&=&\Big|\int_0^T\int_{\Omega}\widetilde{\mu}\varphi\nabla\cdot u_\epsilon\nabla\cdot \mathcal B\big[\rho_\epsilon^\theta-\oint_{\Omega}\rho_\epsilon^\theta\big]\Big|\\
&=&\Big|\int_0^T\int_{\Omega}\widetilde{\mu}\varphi\nabla\cdot u_\epsilon\big(\rho_\epsilon^\theta-\oint_{\Omega}\rho_\epsilon^\theta \big)\Big|\\
&\leq& C\Big\|u_\epsilon\Big\|_{L^2([0,T],H^1(\Omega))}\Big\|\rho_\epsilon\Big\|_{L^\infty([0,T],L^{2\theta}(\Omega))}^\theta
\eeqno
is uniformly bounded, provided $\theta\leq\frac{\gamma}{2}.$

\noindent(5) For $L_5^m$, we have
\beqno
\big|L_5^m\big|&=&(1-\theta)\Big|\int_0^T\int_{\Omega}\varphi\rho_{\epsilon} u_{\epsilon}\mathcal B\big[\rho_\epsilon^\theta\nabla\cdot u_{\epsilon}-\oint_{\Omega}\rho_\epsilon^\theta\nabla\cdot u_{\epsilon}\big]\Big|\\
&\leq& C\int_0^T\big\|\rho_\epsilon\big\|_{L^\gamma(\Omega)}\big\|u_\epsilon\big\|_{L^6(\Omega)}
\big\|\mathcal B\big[\rho_\epsilon^\theta\nabla\cdot u_{\epsilon}-\oint_{\Omega}\rho_\epsilon^\theta\nabla\cdot u_{\epsilon}\big]\big\|_{L^{\frac{6\gamma}{5\gamma-6}}(\Omega)}dt\\
&\le& C\int_0^T\big\|\rho_\epsilon\big\|_{L^\gamma(\Omega)}\big\|u_\epsilon\big\|_{L^6(\Omega)}
\big\|\rho_\epsilon^\theta\nabla\cdot u_{\epsilon}\big\|_{L^{\frac{6\gamma}{7\gamma-6}}(\Omega)}dt\\
&\le& C\int_0^T\big\|\rho_\epsilon\big\|_{L^\gamma(\Omega)}\big\|u_\epsilon\big\|_{H^1(\Omega)}
\big\|\rho_\epsilon\big\|_{L^{\frac{3\gamma\theta}{2\gamma-3}}(\Omega)}^\theta
\big\|\nabla\cdot u_{\epsilon}\big\|_{L^{2}(\Omega)}dt
\eeqno
is uniformly bounded, provided $\theta<\frac{2\gamma}3-1$.

\noindent (6)  For $L_6^m,$ we have
\beqno
\big|L_6^m\big|&=&\Big|\int_0^T\int_{\Omega}\varphi\big(\triangle d_{\epsilon}+\frac{1}{\epsilon^2}(1-|d_{\epsilon}|^2)d_{\epsilon}\big)\cdot\nabla d_{\epsilon}\mathcal B\big[\rho_{\epsilon}^{\theta}-\oint_{\Omega}\rho_{\epsilon}^\theta \big]\Big|\\
&\leq& C\Big\|\triangle d_{\epsilon}+\frac{1}{\epsilon^2}(1-|d_{\epsilon}|^2)d_{\epsilon}\Big\|_{L^2(Q_T)}\cdot\\
&&\Big\|\nabla d_{\epsilon}\Big\|_{L^\infty([0,T],L^2(\Omega))}
\Big\|\mathcal B\big[\rho_{\epsilon}^{\theta}-\oint_{\Omega}\rho_{\epsilon}^\theta\big]\Big\|_{L^2([0,T],L^\infty(\Omega))}\\
&\leq& C\Big\|\big\|\mathcal B\big[\rho_{\epsilon}^{\theta}-\oint_{\Omega}\rho_{\epsilon}^\theta \big]\big\|_{W^{1,\frac{\gamma}{\theta}}(\Omega)}\Big\|_{L^2([0,T])}\\
&\leq& C\sup\limits_{0<t\leq T}\Big\|\rho_\epsilon^\theta\Big\|_{L^{\frac{\gamma}{\theta}}(\Omega)}
\leq C\Big\|\rho_{\epsilon}\Big\|_{L^\infty([0,T], L^{\gamma}(\Omega))}^\theta,
\eeqno
provided $\theta<\frac{\gamma}{3}$, where we have used the energy inequality (\ref{global_bd})
and the Sobolev embedding theorem $W^{1,\frac{\gamma}{\theta}}(\Omega)\subset L^\infty(\Omega)$.

\noindent (7) For $L_7^m$, we have
\beqno
\big|L_7^m\big|&=&\Big|\int_0^T\int_{\Omega}|\varphi\rho_\epsilon u_\epsilon
\mathcal B\big[\nabla\cdot(\rho_\epsilon^\theta u_\epsilon)\big]\Big|\\
&\leq& C\int_0^T\big\|\rho_\epsilon\big\|_{L^\gamma(\Omega)}\big\|u_\epsilon\big\|_{L^6(\Omega)}
\big\|\mathcal B\big[\nabla\cdot(\rho_\epsilon^\theta u_\epsilon)\big]\big\|_{L^{\frac{6\gamma}{5\gamma-6}}(\Omega)}dt\\
&\leq& C\int_0^T\big\|\rho_\epsilon\big\|_{L^\gamma(\Omega)}
\big\|u_\epsilon\big\|_{L^6(\Omega)}\big\|\rho_\epsilon^\theta u_\epsilon\big\|_{L^{\frac{6\gamma}{5\gamma-6}}(\Omega)}dt\\
&\leq& C\int_0^T\big\|\rho_\epsilon\big\|_{L^\gamma(\Omega)}
\big\|u_\epsilon\big\|_{L^6(\Omega)}^2\big\|\rho_\epsilon\|_{L^{\frac{3\gamma\theta}{2\gamma-3}}(\Omega)}^\theta dt
\eeqno
is uniformly bounded, provided $\theta<\frac{2\gamma}{3}-1.$

It is clear that we can choose sufficiently small $\theta>0$ depending only on $\gamma$ such that
all these estimates on $L_i^m,$  $i=1,\cdots,7,$ hold.
Therefore, by putting together (1), $\cdots$, (7), we obtain the estimate (\ref{rho-bd}).

\subsection{$\rho_{\epsilon}u_{\epsilon}\otimes u_\epsilon\rightarrow\rho u\otimes u$ in the sense of distributions}

As $\rho_{\epsilon}u_{\epsilon}$ is bounded in
$L^{\infty}([0,T], L^{\frac{2\gamma}{\gamma+1}}(\Omega))$, it follows from the section 5.3 that
\beq\label{4.22}
\rho_{\epsilon}u_{\epsilon}\rightharpoonup\rho u\ \mbox{weak}^*\ \mbox{in}\ L^{\infty}([0,T], L^{\frac{2\gamma}{\gamma+1}}(\Omega)).
\eeq
Meanwhile,
since \beqno
&&\partial_t(\rho_\epsilon u_{\epsilon})=-\nabla\cdot(\rho_{\epsilon}u_{\epsilon}\otimes u_{\epsilon})+a\nabla\rho_\epsilon^{\gamma}\\
&&\qquad\qquad\ \ +\mathcal Lu_\epsilon-\nabla\cdot\Big[\nabla d_\epsilon\odot\nabla d_\epsilon-\big(\frac12|\nabla d_\epsilon|^2
+\frac{(1-|d_\epsilon|^2)^2}{4\epsilon^2}\big)\mathbb I_3\Big]\\
&&\qquad\qquad\in L^2([0,T],W^{-1,\frac{6\gamma}{4\gamma+3}}(\Omega))
+L^{\frac{\gamma+\theta}{\gamma}}([0,T],W^{-1,\frac{\gamma+\theta}{\gamma}}(\Omega))\\
&&\qquad\qquad\ \ +L^2([0,T], H^{-1}(\Omega))+L^\infty([0,T], W^{-2, \frac54}(\Omega)),
\eeqno
we have that
\beq
\label{4.23}
\rho_\epsilon u_{\epsilon}\rightarrow\rho u\ \mbox{in}\ C([0,T],L^{\frac{2\gamma}{\gamma+1}}_{\rm{weak}}(\Omega))\ {\rm{and}}\ \rho_\epsilon u_{\epsilon}\rightarrow\rho u\ \mbox{in}\ C([0,T],H^{-1}(\Omega)).
\eeq
Hence we obtain
\beq\label{4.24}\rho_\epsilon u_{\epsilon}\otimes u_{\epsilon}\rightarrow\rho u\otimes u
\ \ {\rm{in}}\ \ \mathcal D'(Q_T).
\eeq
It follows from \S5.1, \S5.2, \S5.3, and \S5.4
that, after sending $\epsilon\rightarrow 0^+$ in the equation
(\ref{3.3}), $(\rho,u,d)$  satisfies the following system:
\beq\label{limit_eqn3.1}
\begin{cases}
\partial_t\rho+\nabla\cdot(\rho u)=0,\\
\partial_t(\rho u)+\nabla\cdot(\rho u\otimes u)+a\nabla\overline{\rho^{\gamma}}
=\mathcal{L}u-\nabla\cdot\big[\nabla d\odot\nabla d-\frac12|\nabla d|^2\mathbb{I}_3\big],\\
\partial_t d+u\cdot\nabla d=\triangle d+|\nabla d|^2d,
\end{cases} \eeq
in the sense of distributions,
where $\overline{\rho^\gamma}$ is a weak limit of $\rho_\epsilon^\gamma$ in $L^{\frac{\gamma+\theta}{\gamma}}(Q_T)$.

It is straightforward that $(\rho, u, d)$ satisfies the first two equations of (\ref{limit_eqn3.1}).
To see $(u,d)$ solves the third equation of (\ref{limit_eqn3.1}), we employ the standard technique,
due to Chen \cite{Chen}, as follows. Let $\times$ denote the cross product
in $\mathbb R^3$. Then the equation $(\ref{3.3})_3$ for $(u_\epsilon, d_\epsilon)$ can be rewritten
as
$$(\partial_t d_\epsilon+u_\epsilon\cdot\nabla d_\epsilon)\times d_\epsilon=\Delta d_\epsilon\times d_\epsilon,
\ {\rm{in}}\ \mathcal D'(Q_T).$$
After taking $\epsilon\rightarrow 0$, we have that $(u,d)$ satisfies
\begin{equation}\label{d-eqn}
(\partial_t d+u\cdot\nabla d)\times d=\Delta d\times d, \ {\rm{in}}\ \mathcal D'(Q_T).
\end{equation}
Since $|d|=1$, the equation (\ref{d-eqn}) is equivalent to $(\ref{limit_eqn3.1})_3$.

In order to identify $\overline{\rho^\gamma}$, we need to establish the strong convergence of
$\rho_\epsilon$ to $\rho$ in $L^\gamma(Q_T)$. To do it, we need to have fine
estimates of the effective viscous flux, which has played important rules in the study of compressible
Navier-Stokes equations (see \cite{F1} and \cite{L1}).

For $k\ge 1$, define $T_k(z)=kT(\frac{z}{k}):\mathbb R\to\mathbb R$,
where $T(z)\in C^\infty(\mathbb R)$ is a concave function such that
\beqno T(z)=
  \begin{cases}
    z, & \mbox{if}\ z\leq 1, \\
    2, & \mbox{if}\ z\geq 3.
  \end{cases}
\eeqno

\subsection{Fine estimates of effective viscous flux $H_\epsilon:=a\rho_\epsilon^\gamma-\widetilde{\mu}\nabla\cdot u_\epsilon$}
For any fixed $k\ge 1$, there holds
\beq\label{4.27}
\lim\limits_{\epsilon\rightarrow0}\int_0^T\int_{\Omega}\psi\phi
\big(a\rho_\epsilon^\gamma-\widetilde{\mu}\nabla\cdot u_\epsilon\big)T_k(\rho_\epsilon)
=\int_0^T\int_{\Omega}\psi\phi\big(a(\overline{\rho^\gamma})-\widetilde{\mu}\nabla\cdot u\big)\overline{T_k(\rho)},
\eeq
for any $\psi\in C_0^\infty((0,T))$ and $\phi\in C_0^\infty(\Omega)$.
By density arguments, similar to \cite{F1}, it is not hard to see that (\ref{4.27}) remains to be true for $\phi=\psi=1$.

Since $(\rho_\epsilon,u_\epsilon)$ is a renormalized solution to $(\ref{3.3})_1$ in $Q_T$, it is clear
that if we extend $(\rho_\epsilon, u_\epsilon)$ to $\mathbb R^3$ by letting it to be zero in $\mathbb R^3\setminus
\Omega$, then $(\rho_\epsilon, u_\epsilon)$ is also a renormalized solution of $(\ref{3.3})_1$ in $\mathbb R^3$.
Replacing $b(z)$ by $T_k(z)$ in $(\ref{3.3})_1$  yields
\begin{equation}\label{tk-rho}
\partial_t \big(T_k(\rho_\epsilon)\big)
+\nabla\cdot(T_k(\rho_\epsilon)u_\epsilon)+\big(T'(\rho_\epsilon)\rho_\epsilon-T_k(\rho_\epsilon)\big)
\nabla\cdot u_\epsilon=0\  \mbox{in}\ \mathcal D'(\mathbb R^3\times (0,T)).
\end{equation}
Since $T_k(\rho_\epsilon)$ is bounded in $L^\infty(Q_T)$, we have
$$T_k(\rho_\epsilon)\rightharpoonup\overline{T_k(\rho)}\ \ \mbox{weak}^* \mbox{in}\ L^\infty(Q_T).
$$
This, combined with the equation (\ref{tk-rho}), implies that for any $p\in (1,+\infty),$
\beq\label{4.29}
T_k(\rho_\epsilon)\rightarrow \overline{T_k(\rho)}\ \ \mbox{in}\ C([0,T],L^p_{\rm{weak}}(\Omega))
\ \ \mbox{and\ in}\ \ C([0,T],H^{-1}(\Omega)).
\eeq
Hence, after sending $\epsilon\rightarrow0$ in the equation (\ref{tk-rho}), we have
\beq\label{4.30}
\partial_t\overline{T_k(\rho)}+\nabla\cdot(\overline{T_k(\rho)}u)+\overline{\big(T'_k(\rho)\rho-T_k(\rho)\big)\nabla\cdot u}=0\ \ \ \mbox{in}\ \ \mathcal D'(Q_T),
\eeq
where $\overline{\big(T'_k(\rho)\rho-T_k(\rho)\big)\nabla\cdot u}$ is a weak limit of
$\big(T_k'(\rho_\epsilon)\rho_\epsilon-T_k(\rho_\epsilon)\big)\nabla\cdot u_\epsilon$  in $\ L^2(Q_T)$.

Now we need to estimate the effective viscous flux $(a\rho_\epsilon^\gamma-\widetilde{\mu}\nabla\cdot u_\epsilon)$.
Define the operator $\mathcal{A}=(\mathcal{A}_1,\mathcal{A}_2,\mathcal{A}_3)$  by letting $$\mathcal{A}_i=\partial_{x_i}\triangle^{-1}$$
for $i=1,2,3$, where $\triangle^{-1}$ denote the inverse of the Laplace operator on $\mathbb R^3$
(see \cite{FNP1}). By the $L^p$ regularity theory of the Laplace equation,
we have
\beq\label{4.31}
\begin{cases}\big\|\mathcal{A} v\big\|_{W^{1,s}(\Omega)}\leq C\big\|v\big\|_{L^s(\mathbb R^3)}, & \ 1<s<+\infty,\\
\big\|\mathcal{A} v\big\|_{L^q(\Omega)}\leq
C\big\|v\big\|_{L^s(\mathbb R^3)}, &\frac{1}{q}\geq\frac{1}{s}-\frac13, \\
\big\|\mathcal{A}v\big\|_{L^\infty(\Omega)}\leq C\big\|v\big\|_{L^s(\mathbb R^3)},
& s>3,
\end{cases}
\eeq
where $C>0$ depends only on $s$ and $\Omega.$

Testing the equation (\ref{tk-rho}) by $\mathcal{A}_i[\varphi]$ for $\varphi\in C_0^\infty(Q_T)$ yields
\beq\label{4.32}
\partial_t \big(\mathcal{A}_i[T_k(\rho_\epsilon)]\big)
+\nabla\cdot\big(\mathcal{A}_i[(T_k(\rho_\epsilon)u_\epsilon)]\big)
+\mathcal{A}_i\Big[(T_k'(\rho_\epsilon)\rho_\epsilon-T_k(\rho_\epsilon))\nabla\cdot u_\epsilon\Big]=0,
\eeq
in $\mathcal D'(\mathbb R^3\times (0,T))\cap L^2(\mathbb R^3\times (0,T)).$
This implies $\partial_t \big(\mathcal{A}_i[T_k(\rho_\epsilon)]\big)\in L^2(\mathbb R^3\times (0,T))$.
Hence we can test the equation $(\ref{3.3})_2$ by $\psi\phi\mathcal{A}[T_k(\rho_\epsilon)]$,
for $\phi\in C^\infty_0(\Omega)$ and $\psi\in C_0^\infty((0,T))$, and obtain
\beq\nonumber&&\int_0^T\int_{\Omega}\psi\phi(a\rho_\epsilon^\gamma-\widetilde{\mu}\nabla\cdot u_\epsilon)T_k(\rho_\epsilon)\\
&&=\int_0^T\int_\Omega\psi (\widetilde{\mu}\nabla\cdot u_\epsilon-a\rho_\epsilon^\gamma)\nabla\phi \mathcal A[T_k(\rho_\epsilon)]\nonumber\\
&&\nonumber\ +\mu\int_0^T\int_{\Omega}\psi\big\{\nabla\phi\nabla u_\epsilon^i\mathcal{A}_i[T_k(\rho_\epsilon)]-u_\epsilon^i\nabla\phi\nabla\mathcal{A}_i[T_k(\rho_\epsilon)]+u_\epsilon\nabla\phi T_k(\rho_\epsilon)\big\}\\
&&\nonumber\ -\int_0^T\int_{\Omega}\phi\rho_\epsilon u_\epsilon\big\{\partial_t\psi\mathcal{A}[T_k(\rho_\epsilon)]+\psi\mathcal{A}[(T_k(\rho_\epsilon)-T'_k(\rho_\epsilon)\rho_\epsilon)\nabla\cdot u_{\epsilon}]\big\}\\
\nonumber&&\ -\int_0^t\int_{\Omega}\psi\rho_\epsilon u_\epsilon^iu_\epsilon^j\partial_j\phi\mathcal{A}_i[T_k(\rho_\epsilon)]\\
\nonumber&&\ +\int_0^T\int_{\Omega}\psi u_\epsilon^i\big\{T_k(\rho_\epsilon)\mathcal R_{ij}[\phi\rho_\epsilon u_\epsilon^i]-\phi\rho_\epsilon u_\epsilon^j\mathcal R_{ij}[T_k(\rho_\epsilon)]\big\}\\
\label{4.33}&&
\ -\int_0^T\int_{\Omega}\psi\nabla\cdot\Big\{\nabla d_\epsilon\odot\nabla d_\epsilon-[\frac12|\nabla d_\epsilon|^2+\frac{1}{4\epsilon^2}(1-|d_\epsilon|^2)^2]\mathbb{I}_3\Big\}\cdot
\phi\mathcal{A}[T_k(\rho_\epsilon)],
\eeq
where $\mathcal R_{ij}=\partial_{x_j}\mathcal A_i$ is the Riesz transform.

Similarly, we can test the equation $(\ref{limit_eqn3.1})_2$ by
$\psi\phi\mathcal{A}_i[T_k(\rho)]$ and obtain
\beq\label{4.34}
\nonumber&&\int_0^T\int_{\Omega}\psi\phi(a\rho^\gamma-\widetilde{\mu}\nabla\cdot u)T_k(\rho)\nonumber\\
&&=\int_0^T\int_\Omega\psi\big(\widetilde{\mu}\nabla\cdot u-a\overline{\rho^\gamma}\big)
\nabla\phi \mathcal A[T_k(\rho)]\nonumber\\
&&\nonumber\ +\mu\int_0^T\int_{\Omega}\psi\big\{\nabla\phi\nabla u^i\mathcal{A}_i[T_k(\rho)]-u^i\nabla\phi\nabla\mathcal{A}_i[T_k(\rho)]+u\nabla\phi T_k(\rho)\big\}\\
&&\nonumber\ -\int_0^T\int_{\Omega}\phi\rho u\big\{\partial_t\psi\mathcal{A}[T_k(\rho)]
+\psi\mathcal{A}[(T_k(\rho)-T'_k(\rho)\rho)\nabla\cdot u]\big\}\\
\nonumber&&\ -\int_0^t\int_{\Omega}\psi\rho u^iu^j\partial_j\phi\mathcal{A}_i[T_k(\rho)]\\
\nonumber&&\ +\int_0^T\int_{\Omega}\psi u^i\big\{T_k(\rho)\mathcal R_{ij}[\phi\rho u^i]
-\phi\rho u^j\mathcal R_{ij}[T_k(\rho)]\big\}\\
&&\ -\int_0^T\int_{\Omega}\psi\nabla\cdot\Big[\nabla d\odot\nabla d-\frac12|\nabla d|^2\mathbb{I}_3\Big]\cdot \phi\mathcal{A}[T_k(\rho)].
\eeq
To prove (\ref{4.27}), it suffices to show that each term in the right hand side of (\ref{4.33}) converges to the
corresponding term  in the right hand side of (\ref{4.34}).
Since the convergence of the first five terms in the right hand side of (\ref{4.33})
can be done in the exact same way as in \cite{FNP1} (see also \cite{WY1}),
we only indicate how to show the convergence of the last term in the right hand side of (\ref{4.33}), namely,
\beq\label{4.35}
\nonumber&&\int_0^T\int_{\Omega}\psi\nabla\cdot\Big[\nabla d_\epsilon\odot\nabla d_\epsilon-(\frac12|\nabla d_\epsilon|^2+\frac{1}{4\epsilon^2}(1-|d_\epsilon|^2)^2)\mathbb{I}_3\Big]
\cdot \phi\mathcal{A}[T_k(\rho_\epsilon)]\\
&&\rightarrow\int_0^T\int_{\Omega}\psi\nabla\cdot\Big[\nabla d\odot\nabla d-\frac12|\nabla d|^2\mathbb{I}_3\Big]\cdot\phi\mathcal{A}[T_k(\rho)] \ {\rm{as}}\ \epsilon\rightarrow 0.
\eeq
To see this, first observe that  $T_k(\rho_\epsilon)$ is bounded
in $L^\infty(Q_T)$ and hence we have, by (\ref{4.32}), that
(see also \cite{FNP1})
\beq\label{4.36}
\mathcal{A}[T_k(\rho_\epsilon)]\rightarrow\mathcal{A}[T_k(\rho)]
\ {\rm{in}}\ C(\overline\Omega\times [0, T]).
\eeq
Secondly, observe that a.e. in $Q_T$, there holds
\begin{eqnarray*}
\nabla\cdot\Big[\nabla d_\epsilon\odot\nabla d_\epsilon-(\frac12|\nabla d_\epsilon|^2+\frac{1}{4\epsilon^2}(1-|d_\epsilon|^2)^2)\mathbb{I}_3\Big]&=&\big(\Delta d_\epsilon+\frac{1}{\epsilon^2}(1-|d_\epsilon|^2)d_\epsilon\big)\cdot\nabla d_\epsilon\\
&=&\big(\partial_t d_\epsilon+u_\epsilon\cdot\nabla d_\epsilon\big)\cdot\nabla d_\epsilon.
\end{eqnarray*}
By the energy inequality (\ref{global_bd}), we see that
$\big(\partial_t d_\epsilon+u_\epsilon\cdot\nabla d_\epsilon\big)$ is bounded in $L^2(Q_T)$ and hence
there exists $v\in L^2(Q_T)$ such that
\begin{equation}\label{weak1}
\big(\partial_t d_\epsilon+u_\epsilon\cdot\nabla d_\epsilon\big)\rightharpoonup v
\ {\rm{in}}\ L^2(Q_T).
\end{equation}
On the other hand, since $d_\epsilon\rightarrow d$ in $L^2([0,T], H^1_{\rm{loc}}(\Omega))$ and $u_\epsilon\rightharpoonup
u$ in $L^2([0,T], H^1_{0}(\Omega))$, we have that
$$\big(\partial_t d_\epsilon+u_\epsilon\cdot\nabla d_\epsilon\big)\rightarrow
 \big(\partial_t d+u\cdot\nabla d\big)
\ {\rm{in}}\ \mathcal D'(Q_T).$$
Hence we have
\begin{equation}
\label{weak2}
v=\partial_t d+u\cdot\nabla d\ \ {\rm{in}}\ \ Q_T.
\end{equation}
By (\ref{4.36}) and the local $L^2$-convergence of $\nabla d_\epsilon$ to $\nabla d$ in $Q_T$, we know that
$$\nabla d_\epsilon \phi \mathcal A[T_k(\rho_\epsilon)]\rightarrow \nabla d \phi \mathcal A[T_k(\rho)]
\ \ \ {\rm{in}}\ \ \  L^2(Q_T).$$
Hence we obtain
\beq\label{4.39}
\nonumber&&\int_0^T\int_{\Omega}\psi\nabla\cdot\Big[\nabla d_\epsilon\odot\nabla d_\epsilon-(\frac12|\nabla d_\epsilon|^2+\frac{1}{4\epsilon^2}(1-|d_\epsilon|^2)^2)\mathbb{I}_3\Big]
\cdot \phi\mathcal{A}[T_k(\rho_\epsilon)]\\
&&\rightarrow\int_0^T\int_{\Omega}\psi (\partial_t d+u\cdot\nabla d)\nabla d\cdot\phi \mathcal A[T_k(\rho)]
\ {\rm{as}}\ \epsilon\rightarrow 0.
\eeq
Applying the equation $(\ref{limit_eqn3.1})_3$ and the fact that $\langle|\nabla d|^2 d,\nabla d\rangle=0$ a.e. in $Q_T$,
we obtain
\begin{eqnarray}\label{4.40}
&&\int_0^T\int_{\Omega}\psi (\partial_t d+u\cdot\nabla d)\nabla d\cdot\phi \mathcal A[T_k(\rho)]\nonumber\\
&=&\int_0^T\int_{\Omega}\psi (\Delta d+|\nabla d|^2 d)\nabla d\cdot\phi \mathcal A[T_k(\rho)]\nonumber\\
&=&\int_0^T\int_{\Omega}\psi \Delta d\nabla d\cdot\phi \mathcal A[T_k(\rho)]\nonumber\\
&=&\int_0^T\int_\Omega \psi\nabla\cdot\Big[\nabla d\odot\nabla d-\frac12|\nabla d|^2 \mathbb I_3\Big]\cdot
\phi\mathcal A[T_k(\rho)].
\end{eqnarray}
It is easy to see that (\ref{4.35}) follows from (\ref{4.39}) and (\ref{4.40}).

In order to show the strong convergence of $\rho_\epsilon$, we also need to estimate on the oscillation defect measure
of ($\rho_\epsilon-\rho$) in $L^\gamma(Q_T)$.

\subsection{Estimate of oscillation of defect measures}
{There exists $C>0$ such that for any $k\geq 1$, there holds
\beq\label{4.41}
\limsup\limits_{\epsilon\rightarrow0}\Big\|T_k(\rho_\epsilon)-T_k(\rho)\Big\|_{L^{\gamma+1}(Q_T)}^{\gamma+1}\leq\lim\limits_{\epsilon\rightarrow0}\int_0^T\int_{\Omega}\Big[\rho_\epsilon^\gamma T_k(\rho_\epsilon)-\overline{\rho^\gamma}\overline{T_k(\rho)}\Big]\leq C,
\eeq
where $\overline{T_k(\rho)}$ is a weak$^*$ limit of $T_k(\rho_\epsilon)$ in $L^\infty(Q_T)$.}

Following the lines of argument presented in  \cite{FNP1} and using (\ref{4.27}), we obtain
\begin{eqnarray*}
&&\limsup\limits_{\epsilon\rightarrow0}\Big\|T_k(\rho_\epsilon)-T_k(\rho)\Big\|_{L^{\gamma+1}(Q_T)}^{\gamma+1}\\
&&\leq\lim\limits_{\epsilon\rightarrow0}\int_0^T\int_{\Omega}\Big[\rho_\epsilon^\gamma T_k(\rho_\epsilon)-\overline{\rho^\gamma}\overline{T_k(\rho)}\Big]\\
&&=\frac{\widetilde\mu}{a}\lim_{\epsilon\rightarrow  0}\int_0^T\int_\Omega
\Big[(\nabla\cdot u_\epsilon)\big(T_k(\rho_\epsilon)-\overline{T_k(\rho)}\big)\Big]\\
&&\le C\Big(\sup_{\epsilon>0}\big\|\nabla u_\epsilon\big\|_{L^2(Q_T)}\Big)
\limsup_{\epsilon\rightarrow 0}\Big[\big\|T_k(\rho_\epsilon)-T_k(\rho)\big\|_{L^2(Q_T)}\\
&&\qquad+\big\|T_k(\rho)-\overline{T_k(\rho)}\big\|_{L^2(Q_T)}\Big]\\
&&\le C\limsup_{\epsilon\rightarrow 0}\big\|T_k(\rho_\epsilon)-T_k(\rho)\big\|_{L^2(Q_T)},
\end{eqnarray*}
this implies (\ref{4.41}) by applying Young's inequality and using $\gamma+1>2$.

We now want to show

\medskip
\noindent{\it Claim 1}.  {\it $(\rho,u)$  is a renormalized solution to the equation $(\ref{limit_eqn3.1})_1$.}

Observe that (\ref{4.30}) also holds for $(\rho_\epsilon, u_\epsilon)$ in $\mathbb R^3$ provided it is set to be zero
in $\mathbb R^3\setminus\Omega$.
Hence we have
\beq\label{4.42}
\partial_t\overline{T_k(\rho)}+\nabla\cdot(\overline{T_k(\rho)}u)+\overline{\big(T'_k(\rho)\rho-T_k(\rho)\big)\nabla\cdot u}=0\ \ \ \mbox{in}\ \ \mathcal D'(\mathbb R^3\times (0,T)).
\eeq
As in the step 3, we can mollify (\ref{4.42}) and obtain
\beq\label{4.43}\partial_t\big(S_m\Big[\overline{T_k(\rho)}\Big]\big)
+\nabla\cdot\big (S_m\Big[\overline{T_k(\rho)}\Big]u\big)+S_m\Big[\overline{[T'_k(\rho)\rho-T_k(\rho)]\nabla\cdot u}\Big]
=q_m,\eeq
where
$$q_m:=\nabla\cdot\big(S_m\Big[\overline{T_k(\rho)}\Big]u\big)-S_m\Big[\nabla\cdot(\overline{T_k(\rho)}u)\Big]
\rightarrow 0 \ {\rm{in}}\ L^2([0,T],L^s(\Omega)),\ \ {\rm{as}}\ m\rightarrow \infty,
$$
for any $s\in[1,2)$, by virtue of Lemma 2.3 in \cite{L1}.

Let $b$ be a test function in the definition of renormalized solutions of $(\ref{limit_eqn3.1})_1$.
Multiplying (\ref{4.43}) by $b'\Big(S_m\big[\overline{T_k(\rho)}\big]\Big)$
yields
\begin{eqnarray*}
&&\partial_t\big(b(S_m\Big[\overline{T_k(\rho)}\Big])\big)
+\nabla\cdot\big(b\big(S_m\Big[\overline{T_k(\rho)}\Big]\big)u\big)
+\Big(b'(S_m\Big[\overline{T_k(\rho)}\Big])S_m\Big[\overline{T_k(\rho)}\Big]
-b(S_m\Big[\overline{T_k(\rho)}\Big])\Big)\nabla\cdot u\\
&&=-b'(S_m\Big[\overline{T_k(\rho)}\Big])S_m\Big(\overline{[T'(\rho)\rho-T_k(\rho)]\nabla\cdot u}\Big)
+b'(S_m\Big[\overline{T_k(\rho)}\Big])q_m.
\end{eqnarray*}
Sending $m\rightarrow+\infty$ in the above equation yields
that
\begin{eqnarray}
&&\nonumber
\partial_t\big(b\Big(\overline{T_k(\rho)}\Big)\big)
+\nabla\cdot\big(b\Big(\overline{T_k(\rho)}\Big)u\big)
+\Big(b'\Big(\overline{T_k(\rho)}\Big)\overline{T_k(\rho)}-b\Big(\overline{T_k(\rho)}\Big)\Big)\nabla\cdot u\\
&&=-b'\Big(\overline{T_k(\rho)}\Big)\overline{\big[T'(\rho)\rho-T_k(\rho)\big]\nabla\cdot u}
\ \ \ \ {\rm{in}}\ \ \ \mathcal D'\big(\mathbb R^3\times (0,T)\big). \label{rho-renorm1}
\end{eqnarray}
On the other hand, for $p\in [1,\gamma)$, we have
\beq\label{4.44}
\big\|\overline{T_k(\rho)}-\rho\big\|_{L^p(Q_T)}^p
\leq\liminf\limits_{\epsilon\rightarrow 0}
\big\|T_k(\rho_\epsilon)-\rho_\epsilon\big\|_{L^p(Q_T)}^p.
\eeq
On the other hand,
we have
\beq\big\|T_k(\rho_\epsilon)-\rho_\epsilon\big\|_{L^p(Q_T)}^p
&\leq&\int_{\{\rho_\epsilon\geq k\}}\big|kT(\frac{\rho_\epsilon}{k})-\rho_\epsilon\big|^p\nonumber\\
&\leq& 2^p\int_{\{\rho_\epsilon\geq k\}}\big|\rho_\epsilon\big|^p \ \big({\rm{since}}\ \displaystyle
kT(\frac{\rho_\epsilon}{k})\leq \rho_\epsilon \big)\nonumber\\
&\leq&2^pk^{-\gamma+p}\int_{\{\rho_\epsilon\geq k\}}\rho_\epsilon^\gamma\nonumber\\
&\leq& Ck^{-\gamma+p}\rightarrow0,\  \mbox{as}\ k\rightarrow+\infty, \mbox{uniformly\ in}\ \epsilon.
\label{4.45}
\eeq
It follows from (\ref{4.44}) and (\ref{4.45})  that
\begin{equation}\label{4.46}
\lim\limits_{k\rightarrow+\infty}\Big\|\overline{T_k(\rho)}-\rho\Big\|_{L^{p}(Q_T)}=0,
\ \mbox{for}\ p\in[1,\gamma).
\end{equation}
For any $M>0$ so large that $b'(z)=0$ for $z\ge M$,  we set
$$Q_{k,M}:=\Big\{(x,t)\in Q_T\ \big|\ \overline{T_k(\rho)}\leq M\Big\}.
$$
Then \beq\nonumber&&\int_0^T\int_{\Omega}\Big|b'\Big(\overline{T_k(\rho)}\Big)
\overline{\big[T'_k(\rho)\rho-T_k(\rho)\big]\nabla\cdot u}\Big|\\
&&\nonumber=\int_{Q_{k,M}}\Big|b'\Big(\overline{T_k(\rho)}\Big)\overline{\big[T'_k(\rho)\rho-T_k(\rho)\big]\nabla\cdot u}\Big|\\
&&\nonumber\leq\sup\limits_{Q_{k,M}}\Big|b'(\overline{T_k(\rho)})\Big|
\int_{Q_{k,M}}\Big|\overline{\big[T^{'}_k(\rho)\rho-T_k(\rho)\big]\nabla\cdot u}\Big|\\
&&\nonumber\leq\sup\limits_{0\leq z\leq M}\big|b'(z)\big|\liminf\limits_{\epsilon\rightarrow0}
\int_{Q_{k,M}}\Big|\big[T'_k(\rho_\epsilon)\rho_{\epsilon}-T_k(\rho_\epsilon)\big]\nabla\cdot u_\epsilon\Big|\\
&&\nonumber\leq C\liminf\limits_{\epsilon\rightarrow0}\big\|\nabla u_\epsilon\big\|_{L^2(Q_T)}
\big\|T'_k(\rho_\epsilon)\rho_{\epsilon}-T_k(\rho_\epsilon) \big\|_{L^2(Q_{k,M})}\\
&&\leq C\liminf\limits_{\epsilon\rightarrow0}
\big\|T'_k(\rho_\epsilon)\rho_{\epsilon}-T_k(\rho_\epsilon) \big\|_{L^1(Q_T)}^{\frac{1}{2}-\frac{1}{2\gamma}}
\big\|T'_k(\rho_\epsilon)\rho_{\epsilon}-T_k(\rho_\epsilon) \big\|_{L^{\gamma+1}(Q_{k,M})}^{\frac{1}{2}+\frac{1}{2\gamma}}.
\label{4.47}
\eeq
Now we can estimate
\beq\label{4.48}
\big\|T'_k(\rho_\epsilon)\rho_{\epsilon}-T_k(\rho_\epsilon)\big\|_{L^1(Q_T)}
\leq 2k^{1-\gamma}\sup\limits_{\epsilon}\|\rho_\epsilon\|_{L^{\gamma}}^\gamma\leq Ck^{1-\gamma}\rightarrow 0,
\ {\rm{as}}\ k\rightarrow\infty.
\eeq
On the other hand, since $T_k(z)$ is a concave function and $T_k''(z)\leq 0$, we have, by Taylor's expansion,
that
$$0=T_k(z)-T'_k(z)z+\frac12T''(\xi z)z^2 \ {\rm{for\ some}} \ \xi\in (0,1).$$
In particular, we have  $T'_k(z)z\leq T_k(z)$ and hence
 \beq\nonumber&&
\big\|T'_k(\rho_\epsilon)\rho_{\epsilon}-T_k(\rho_\epsilon) \big\|_{L^{\gamma+1}(Q_{k,M})}
\leq 2\big\|T_k(\rho_\epsilon)\big\|_{L^{\gamma+1}(Q_{k,M})}\\
&&\nonumber\leq 2\Big(\big\|T_k(\rho_\epsilon)-T_k(\rho)\big\|_{L^{\gamma+1}(Q_{k,M})}
+\big\|T_k(\rho)-\overline{T_k(\rho)}\big\|_{L^{\gamma+1}(Q_{k,M})}
+\big\|\overline{T_k(\rho)}\big\|_{L^{\gamma+1}(Q_{k,M})}\Big)\\
&&\nonumber\leq 2\Big(\big\|T_k(\rho_\epsilon)-T_k(\rho)\big\|_{L^{\gamma+1}(Q_T)}+
\big\|T_k(\rho)-\overline{T_k(\rho)}\big\|_{L^{\gamma+1}(Q_T)}+
M|Q_{k,M}|^{\frac{1}{\gamma+1}}\Big).
\eeq
Applying (\ref{4.41}), we then obtain that there exists $C>0$ independent of $k$ such that
\beq\label{4.49}
\limsup\limits_{\epsilon\rightarrow0}
\big\|T'_k(\rho_\epsilon)\rho_\epsilon-T_k(\rho_\epsilon)\big\|_{L^{\gamma+1}(Q_{k,M})}
\leq C\big(1+M|Q_{k,M}|^{\frac{1}{\gamma+1}}\big)
\leq C.
\eeq
Substituting (\ref{4.48}) and (\ref{4.49}) into (\ref{4.47}) yields
\begin{equation}
\label{4.50}
\lim_{k\rightarrow\infty}\int_0^T\int_{\Omega}\Big|b'\Big(\overline{T_k(\rho)}\Big)
\overline{\big[T'_k(\rho)\rho-T_k(\rho)\big]\nabla\cdot u}\Big|=0.
\end{equation}
Sending $k\rightarrow\infty$ into the equation (\ref{rho-renorm1})
and applying (\ref{4.46}), (\ref{4.50}), we conclude that
$(\rho,u)$ is a renormalized solution of the equation (\ref{limit_eqn3.1}).
This proves Claim 1.

\subsection{$\rho_\epsilon\rightarrow\rho$ strongly in $L^p(Q_T)$
for any $1\le p<\gamma+\theta$. Hence $\overline{\rho^\gamma}=\rho^\gamma$ a.e.
in $Q_T$.}

It suffices to show that $\rho_\epsilon\rightarrow \rho$ in $L^1(Q_T)$. This can be done
in the exactly same lines as in \cite{FNP1}. Here we sketch it for the readers' convenience.
Let $L_k(z)\in C^1(0,+\infty)\cap C([0,+\infty))$ be defined by
\beqno L_k(z)=\begin{cases}z\ln z, & 0\leq z\leq k,\\
                      z\ln k+z\int_{k}^z\frac{T_k(s)}{s^2}ds,  & z>k.
\end{cases}
\eeqno
Note that for $z$ large enough, $L_k(z)$ is a linear function,
i.e., for $z\geq3k,$
$$L_k(z)=\beta_kz-2k,\ \ \mbox{with} \ \ \beta_k=\ln k+\int_k^{3k}\frac{T_k(s)}{s^2}ds+\frac23.$$
Therefore $b_k(z):=L_k(z)-\beta_k z\in C^1(0,+\infty)\cap C([0,+\infty))$ satisfies
$b'_k(z)=0$ for $z$ large enough. Moreover, it is easy to see
$$b'_k(z)z-b_k(z)=T_k(z).$$
Since $(\rho_\epsilon,u_\epsilon)$ is a renormalized solution of the equation $(\ref{3.3})_1$
and $(\rho,u)$ is a renormalized solution of the equation $(\ref{limit_eqn3.1})_2$,
we can take $b(z)=b_k(z)$  in the definition of the renormalized solutions to get that
\begin{equation}
\label{4.52}
\partial_t L_k(\rho_\epsilon)+\nabla\cdot\big(L_k(\rho_\epsilon)u_\epsilon\big)
+T_k(\rho_\epsilon)\nabla\cdot u_\epsilon=0, \ {\rm{in}}\  \mathcal D'(Q_T),
\end{equation}
and \begin{equation}
\label{4.53}\partial_t L_k(\rho)+\nabla\cdot(L_k(\rho)u)+T_k(\rho)\nabla\cdot u=0,
\ {\rm{in}}\ \mathcal D'(Q_T).
\end{equation}
Subtracting (\ref{4.52}) from (\ref{4.53}) gives
\begin{equation}
\label{4.54}\partial_t \big(L_k(\rho_\epsilon)-L_k(\rho)\big)
+\nabla\cdot\big(L_k(\rho_\epsilon)u_\epsilon-L_k(\rho)u\big)+
\big(T_k(\rho_\epsilon)\nabla\cdot u_\epsilon-T_k(\rho)\nabla\cdot u\big)=0,
\end{equation}
in $\mathcal D'(Q_T)$.

Since $L_k(z)$ is a linear function for $z$ sufficiently large,
we have that $L_k(\rho_\epsilon)$ is bounded in $L^\infty([0,T], L^\gamma(\Omega))$,
uniformly in $\epsilon$.
Thus we have
$$L_k(\rho_\epsilon)\rightharpoonup\overline{L_k(\rho)}\ \ \ \mbox{weak}^*\ \mbox{in}\ L^\infty([0,T],
L^\gamma(\Omega)),\ \ \mbox{as}\ \ \epsilon\rightarrow 0.$$
This, combined with the equation (\ref{4.53}), implies
\beq\label{4.55}
L_k(\rho_\epsilon)\rightharpoonup\overline{L_k(\rho)}\ \ \mbox{in}\ C([0,T],L^\gamma_{\rm{weak}}(\Omega))
\cap C([0,T], H^{-1}(\Omega)),\ \ \mbox{as}\ \ \epsilon\rightarrow0.
\eeq
In particular, we have
\begin{equation}
\label{4.56}
L_k(\rho_\epsilon),\ \ L_k(\rho)\in C([0,T], L^\gamma_{\rm{weak}}(\Omega)).
\end{equation}
Hence we can multiply the equation (\ref{4.54}) by $\phi\in C_0^\infty(\Omega)$ and integrate the resulting equation
over $Q_t$, $0<t\le T$, to obtain
\beqno&&\int_{\Omega}[{L_k(\rho_\epsilon)}-L_k(\rho)](t)\phi \\
&&=\int_0^t\int_{\Omega}\big\{[L_k(\rho_\epsilon)u_\epsilon-L_k(\rho)u]\cdot\nabla\phi+[T_k(\rho)\nabla\cdot u-T_k(\rho_\epsilon)\nabla\cdot u_\epsilon]\phi\big\},
\eeqno
where we have used the fact that $[L_k(\rho_\epsilon)-L_k(\rho)]\big|_{t=0}=0$.
Taking $\epsilon\rightarrow0$ in the above equation  yields
\beq\nonumber\int_{\Omega}\Big[\overline{L_k(\rho)}-L_k(\rho)\Big](t)\phi &=&
\int_0^t\int_{\Omega}\Big[\overline{L_k(\rho)}-L_k(\rho)\Big]u\cdot\nabla\phi\\
&&+\lim\limits_{\epsilon\rightarrow0}\int_0^t\int_{\Omega}
\big[T_k(\rho)\nabla\cdot u-T_k(\rho_\epsilon)\nabla\cdot u_\epsilon\big]\phi.
\label{4.57}
\eeq
As in \cite{FNP1}, we can choose $\phi=\phi_m\in C_0^\infty(\Omega)$ in (\ref{4.57}), which
approximates the characteristic function of $\Omega$, i.e.,
\beq\label{4.58}
\begin{cases}0\leq\phi_m\leq1, \ \ \phi_m(x)=1\ \mbox{for}\ x\in\Omega \ {\rm{with\ dist}}(x,\partial\Omega)\geq\frac{1}{m},\\
\phi_m\rightarrow 1 \ {\rm{in}}\ \Omega \ {\rm{as}}\ m\rightarrow\infty
\ {\rm{and}}\ |\nabla\phi_m(x)|\leq2m\ \ \mbox{for\ all\ } x\in\Omega.
\end{cases}
\eeq
We then obtain that for $0<t\le T$, it holds
$$
\int_{\Omega}\Big[\overline{L_k(\rho)}-L_k(\rho)\Big](t)=
\lim\limits_{\epsilon\rightarrow0}\int_0^t\int_{\Omega}
\big[T_k(\rho)\nabla\cdot u-T_k(\rho_\epsilon)\nabla\cdot u_\epsilon\big].
$$
Hence we have
\beq\nonumber&&\int_{\Omega}\Big[\overline{L_k(\rho)}-L_k(\rho)\Big](t)\\
\nonumber&&=\int_0^t\int_{\Omega}T_k(\rho)\nabla\cdot u-\lim\limits_{\epsilon\rightarrow0}\int_0^t\int_{\Omega}T_k(\rho_\epsilon)\nabla\cdot u_\epsilon\\
\nonumber&&=\int_0^t\int_{\Omega}T_k(\rho)\nabla\cdot u+\frac{1}{\widetilde{\mu}}\lim\limits_{\epsilon\rightarrow0}\int_0^t\int_{\Omega}(a\rho^\gamma_\epsilon-\widetilde{\mu}\nabla\cdot u_\epsilon)T_k(\rho_\epsilon)
-\frac{a}{\widetilde{\mu}}\lim\limits_{\epsilon\rightarrow0}\int_0^t\int_{\Omega}\rho_\epsilon^\gamma T_k(\rho_\epsilon)\\
\nonumber&&=\int_0^t\int_{\Omega}T_k(\rho)\nabla\cdot u+\frac{1}{\widetilde{\mu}}\int_0^t\int_{\Omega}\Big(a\overline{\rho^\gamma}-\widetilde{\mu}\nabla\cdot  u\Big)
\overline{T_k(\rho)}-\frac{a}{\widetilde{\mu}}\lim\limits_{\epsilon\rightarrow0}\int_0^t\int_{\Omega}\rho_\epsilon^\gamma T_k(\rho_\epsilon) \big({\rm{by}}\ \big(\ref{4.27})\big)\\
&&\nonumber=\int_0^t\int_{\Omega}\Big[T_k(\rho)-\overline{T_k(\rho)}\Big]\nabla\cdot u-\frac{a}{\widetilde{\mu}}\lim\limits_{\epsilon\rightarrow0}\int_0^t\int_{\Omega}\Big[\rho_\epsilon^\gamma T_k(\rho_\epsilon)-\overline{\rho^\gamma}\overline{T_k(\rho)}\Big]\\
&&\nonumber\leq\int_0^t\int_{\Omega}\Big[T_k(\rho)-\overline{T_k(\rho)}\Big]\nabla\cdot u \ \big({\rm{by}}\ (\ref{4.41})\big)\\
&&\nonumber\leq\Big\|T_k(\rho)-\overline{T_k(\rho)}\Big\|_{L^2(\{\rho\geq k\})}\Big\|\nabla\cdot u\Big\|_{L^{2}(\{\rho\geq k\})}
+\Big\|T_k(\rho)-\overline{T_k(\rho)}\Big\|_{L^2(\{\rho\leq k\})}\Big\|\nabla\cdot u\Big\|_{L^{2}(\{\rho\leq k\})}\\
&&\nonumber\leq C\Big(\Big\|\nabla\cdot u\Big\|_{L^{2}(\{\rho\geq k\})}+
\Big\|T_k(\rho)-\overline{T_k(\rho)}\Big\|_{L^2(\{\rho\leq k\})}\Big)\\
&&\nonumber\leq
C\Big(\Big\|\nabla\cdot u\Big\|_{L^{2}(\{\rho\geq k\})}
+\Big\|T_k(\rho)-\overline{T_{k}(\rho)}\Big\|_{L^1(\{\rho\leq k\})}^{\frac{\gamma-1}{2\gamma}}
\Big\|T_k(\rho)-\overline{T_k(\rho)}\Big\|_{L^{\gamma+1}(\{\rho\leq k\})}^{\frac{\gamma+1}{2\gamma}}\Big)
\\
&&\nonumber\leq
C\Big(\Big\|\nabla\cdot u\Big\|_{L^{2}(\{\rho\geq k\})}+
\Big\|T_k(\rho)-\overline{T_{k}(\rho)}\Big\|_{L^1(\{\rho\leq k\})}^{\frac{\gamma-1}{2\gamma}}\Big),
\label{4.59}
\eeq
where we have used (\ref{4.41}) that guarantees
$$\Big\|T_k(\rho)-\overline{T_k(\rho)}\Big\|_{L^{\gamma+1}(Q_T)}
\le\liminf_{\epsilon\rightarrow 0}\Big\|T_k(\rho_\epsilon)-T_k(\rho)\Big\|_{L^{\gamma+1}(Q_T)}\le C,
$$
uniformly in $k$.

Since $T_k$ is concave, it follows $\overline{T_k(\rho)}\leq T_k(\rho)$.
By the definition of $T_k$,  we also have $T_k(\rho)\leq \rho$. Hence we have
\begin{eqnarray*}
\Big\|T_k(\rho)-\overline{T_{k}(\rho)}\Big\|_{L^1(\{\rho\leq k\})}
&\le& \Big\|\rho-\overline{T_{k}(\rho)}\Big\|_{L^1(\{\rho\leq k\})}\\
&\le &
\Big\|\rho-\overline{T_{k}(\rho)}\Big\|_{L^1(Q_T)}\rightarrow 0
\ {\rm{as}}\ k\rightarrow\infty\ \big({\rm{by}}\ (\ref{4.46})\big)
\end{eqnarray*}
Since $\nabla\cdot u\in L^2(Q_T)$, it follows that
$$\lim_{k\rightarrow\infty}\big\|\nabla\cdot u\big\|_{L^{2}(\{\rho\geq k\})}=0.$$
Therefore we obtain
\beq\label{4.59}
\lim\limits_{k\rightarrow\infty}\int_{\Omega}
\Big[\overline{L_k(\rho)}-L_k(\rho)\Big](t)\leq 0,\  t\in (0,T).
\eeq
It follows from the definition of $L_k$ that
\beq\nonumber
\int_0^T\int_{\Omega}\big|L_k(\rho)-\rho\ln\rho\big|
&\leq&\Big\|L_k(\rho)-\rho\ln\rho\Big\|_{L^1(\{\rho\geq k\})}\\
&\leq&\int\int_{\{\rho\geq k\}}|\rho\ln\rho|\rightarrow0,\ \ \mbox{as}\ k\rightarrow+\infty,
\label{4.60}
\eeq
and
\begin{eqnarray}
\label{4.60}\Big\|L_k(\rho_\epsilon)-\rho_\epsilon\ln\rho_\epsilon\Big\|_{L^1(Q_T)}
&\le&\int\int_{\{\rho_\epsilon\ge k\}}\big|L_k(\rho_\epsilon)-\rho_\epsilon\ln\rho_\epsilon\big|\nonumber\\
&\le& \int\int_{\{\rho_\epsilon\ge k\}}\frac{|L_k(\rho_\epsilon)|+|\rho_\epsilon\ln\rho_\epsilon|}{\rho_\epsilon^\gamma}
\rho_\epsilon^\gamma\nonumber\\
&\le& C(\delta)\int\int_{\{\rho_\epsilon\ge k\}}\frac{\rho_\epsilon^\gamma}
{\rho_\epsilon^{\gamma-1-\delta}} \ \ \big(\delta>0 \ {\rm{is\ sufficiently\ small}} \big)
\nonumber\\
&\leq& Ck^{-\gamma+1+\delta}\rightarrow 0,\  \mbox{as}\ k\rightarrow+\infty,\ \mbox{uniformly\ in}\ \epsilon,
\label{4.61}
\end{eqnarray}
so that by the lower semicontinuity we have
\beq\label{4.62}
\lim_{k\rightarrow\infty}\Big\|\overline{L_k(\rho)}-\overline{\rho\ln\rho}\Big\|_{L^1(Q_T)}
\leq\lim_{k\rightarrow\infty}\liminf\limits_{\epsilon\rightarrow0}
\Big\|L_k(\rho_\epsilon)-\rho_\epsilon\ln\rho_\epsilon\Big\|_{L^1(Q_T)}=0.
\eeq
Combining (\ref{4.59}), (\ref{4.60}), (\ref{4.61}), with (\ref{4.62})  implies
that
$$\int_{\Omega}\Big[\overline{\rho\ln\rho}-\rho\ln\rho\Big](t)\leq 0,\ \ t\in (0,T).$$
Since $\overline{\rho\ln\rho}\ge \rho\ln\rho$ a.e. in $Q_T$,
this implies that
$$\overline{\rho\ln\rho}=\rho\ln\rho\ \ \mbox{a.e.\ in}\ Q_T.
$$
By the convexity of the function $\omega(z)=z\ln z:(0,+\infty)\to\mathbb R$, this implies
 that $$\rho_\epsilon\rightarrow\rho\ {\rm{in}}\ L^1(Q_T).$$
Since $\rho_\epsilon$ is bounded in $L^{\gamma+\theta}(Q_T)$, it follow from
a simple interpolation that
$$\rho_\epsilon\rightarrow \rho  \ {\rm{in}}\ L^p(Q_T) \ {\rm{for\ any}}\ 1\le p<\gamma+\theta.
$$
Thus $\overline{\rho^\gamma}=\rho^\gamma$ a.e. in $Q_T$.

The energy inequality (\ref{1.9}) for $(\rho, u, d)$ follows from the energy inequality (\ref{3.6})
for $(\rho_\epsilon, u_\epsilon, d_\epsilon)$. In fact, (\ref{3.6}) implies
that for almost all $0<t<\infty$, it holds
\begin{equation}
\label{epsilon-ineq}
{\bf F}_\epsilon(t)+\int_0^t\int_\Omega \Big(\mu|\nabla u_\epsilon|^2
+\widetilde\mu|\nabla\cdot u_\epsilon|^2+|\Delta d_\epsilon+\frac{1}{\epsilon^2}(1-|d_\epsilon|^2)d_\epsilon|^2\Big)
\le {\bf F_\epsilon}(0)={\bf E}(0).
\end{equation}
On the other hand, by the lower semicontinuity, we have that for almost all $t\in (0,+\infty)$
\begin{eqnarray}&&{\bf E}(t)+\int_0^t\int_\Omega \Big(\mu|\nabla u|^2
+\widetilde\mu|\nabla\cdot u|^2+|\Delta d+|\nabla d|^2d|^2\Big)\nonumber\\
&&\le\liminf_{\epsilon\rightarrow 0}\Big\{
{\bf F}_\epsilon(t)+\int_0^t\int_\Omega \Big(\mu|\nabla u_\epsilon|^2
+\widetilde\mu|\nabla\cdot u_\epsilon|^2+|\Delta d_\epsilon+\frac{1}{\epsilon^2}(1-|d_\epsilon|^2)d_\epsilon|^2\Big)\Big\},
\label{lsc}
\end{eqnarray}
where we have used the observation that
$$\Delta d_\epsilon+\frac{1}{\epsilon^2}(1-|d_\epsilon|^2)d_\epsilon
=\partial_t d_\epsilon+u_\epsilon\cdot\nabla d_\epsilon
\rightharpoonup \partial_t d+u\cdot\nabla d
=\Delta d+|\nabla d|^2 d\ \ {\rm{in}}\ \ L^2(Q_t).$$
It is clear that (\ref{epsilon-ineq}) and (\ref{lsc}) imply (\ref{1.9}).

After these steps, we conclude that  $(\rho,u,d)$ is a global finite energy weak solution of the system
(\ref{1.1}), under the initial and boundary condition (\ref{1.2}), that satisfies the properties (i)
of Theorem \ref{th:1.1}.

The property (ii) for $(u,d)$ follows from the strong convergence of $d_\epsilon$ to $d$ in $L^2_{\rm{loc}}((0,+\infty), H^1_{\rm{loc}}(\Omega))$.
In fact, it is easy to see that $d_\epsilon\in L^2_{\rm{loc}}((0,+\infty), H^2_{\rm{loc}}(\Omega))$.
For any $X\in C_0^1(\Omega, \mathbb R^3)$ and $\eta\in C_0^1((0,+\infty))$,
we can multiply the equation $(\ref{3.3})_3$ by $\eta(t) X(x)\cdot\nabla d_\epsilon(x)$ and integrate
the resulting equation over $\Omega\times (0,+\infty)$ and apply the integration by parts a few times to obtain
\begin{equation}\label{e-stationary}
\int_0^\infty\eta(t)\int_\Omega \big(e_\epsilon(d_\epsilon) \nabla\cdot X-\nabla d_\epsilon\odot\nabla d_\epsilon:\nabla X\big)
=\int_0^\infty\eta(t)\int_\Omega \big\langle \partial_t d_\epsilon+u_\epsilon\cdot\nabla d_\epsilon, X\cdot\nabla d_\epsilon
\big\rangle,
\end{equation}
where $\displaystyle e_\epsilon(d_\epsilon):=\frac12|\nabla d_\epsilon|^2+\frac{1}{4\epsilon^2}
(1-|d_\epsilon|^2)^2$.
Since $$\partial_t d_\epsilon+u_\epsilon\nabla d_\epsilon
\rightharpoonup \partial_t d+u\cdot\nabla d\ \ {\rm{in}}\ \ L^2(\Omega\times (0,+\infty)),
\ {\rm{as}}\ \epsilon\rightarrow 0,$$
we obtain, by sending $\epsilon \rightarrow 0$ in (\ref{e-stationary}) and applying both
Theorem 3.1 and 5.1, that
\begin{equation}\label{stationary1}
\int_0^T\eta(t)\int_\Omega \big(\frac12|\nabla d|^2\nabla\cdot X-\nabla d\odot\nabla d:\nabla X\big)
=\int_0^T\eta(t)\int_\Omega \big\langle \partial_t d+u\cdot\nabla d, X\cdot\nabla d\big\rangle.
\end{equation}
The proof of Theorem \ref{th:1.1} is now complete.\qed

\medskip

\setcounter{section}{5}
\setcounter{equation}{0}
\section{Large time behavior of finite energy solutions and proof of corollary \ref{long-time}}

In this section, we will study the large time asymptotic behavior of the global finite energy
weak solutions obtained in Theorem
\ref{th:1.1} and give a proof of Corollary \ref{long-time}.

\medskip
\noindent{\bf Proof of Corollary \ref{long-time}}:

\smallskip
First it follows from (\ref{1.9}) that
\begin{equation}\label{global_ineq}
{\rm{esssup}}_{t>0} {\bf E}(t)
+\int_0^\infty\int_\Omega \big(\mu|\nabla u|^2+|\Delta d+|\nabla d|^2 d|^2\big)\le {\bf E}(0).
\end{equation}
For any positive integer $m$, define
$(\rho_m, u_m, d_m):Q_1\to\mathbb R_+\times\mathbb R^3\times\mathbb S^2$ by
$$
\begin{cases}
\rho_m(x,t)=\rho(x, t+m),\\
u_m(x,t)=u(x,t+m),\\
d_m(x,t)=d(x,t+m).
\end{cases}
$$
Then $(\rho_m, u_m, d_m)$ is a sequence of finite energy weak solutions
of (\ref{1.1}) in $Q_1$. It follows from (\ref{global_ineq}) that
\begin{eqnarray}\label{bd1}
&&\big\|\rho_m\big\|_{L^\infty([0,1], L^\gamma(\Omega))}
+\big\|\rho_m^\frac12u_m\big\|_{L^\infty([0,1], L^2(\Omega))}
+\big\|\rho_m u_m\big\|_{L^\infty([0,1], L^{\frac{2\gamma}{2\gamma+1}}(\Omega))}
\nonumber\\
&&\quad+\big\|d_m\big\|_{L^\infty([0,1], H^1(\Omega))}\le C({\bf E}(0)),
\end{eqnarray}
and
\begin{equation}
\label{bd2}
\lim_{m\rightarrow\infty}\int_0^1\Big(\big\|\nabla u_m\big\|_{L^2(\Omega)}^2
+\big\|\Delta d_m+|\nabla d_m|^2 d_m\big\|_{L^2(\Omega)}^2\Big)=0.
\end{equation}
After passing to a subsequence, we may assume that as $m\rightarrow\infty$,
$$
\rho_m\rightharpoonup \rho_\infty \ {\rm{in}}\ L^\gamma(Q_1),
\ u_m\rightharpoonup u_\infty \ {\rm{in}}\ L^2([0,1], H^1_0(\Omega)),
\ d_m\rightharpoonup d_\infty \ {\rm{in}}\ L^2([0,1], H^1(\Omega)).
$$
Applying (\ref{bd2}) and the Poincar\'e inequality, we have
$$
\lim_{m\rightarrow\infty}\int_0^1\big\|u_m\big\|_{L^2(\Omega)}^2=0,
$$
and hence $u_\infty=0$  a.e. in $Q_1$.

Sending $m\rightarrow\infty$ in $(\ref{1.1})_3$, we see that $d_\infty$ solves
$$\partial_t d_\infty=\Delta d_\infty+|\nabla d_\infty|^2 d_\infty  \ {\rm{in}}\ Q_1.$$
On the other hand, by the lower semicontinuity and (\ref{bd2}) we have
$$\int_0^1\int_\Omega \big|\Delta d_\infty+|\nabla d_\infty|^2 d_\infty \big|^2=0.$$
Hence $\partial_t d_\infty=0$ in $Q_1$ and $d_\infty(x,t)=d_\infty(x)\in
H^1(\Omega,\mathbb S^2_+)$ is a harmonic map, with $d_\infty=d_0$ on $\partial\Omega$.

By H\"older's inequality, (\ref{bd1}), and (\ref{bd2}), we have
\begin{equation}
\label{bd3}
\lim_{m\rightarrow\infty}\int_0^1 \Big(\big\|\rho_m u_m\big\|_{L^{\frac{6\gamma}{\gamma+6}}(\Omega)}^2
+\big\|\rho_m |u_m|^2\big\|_{L^{\frac{3\gamma}{\gamma+3}}(\Omega)}^2\Big)=0.
\end{equation}
Since $(\rho_m, u_m,d_m)$ solves $(\ref{1.1})_1$ in $Q_1$, we have
$$\partial_t(\rho_m u_m)+\nabla\cdot(\rho_m u_m\otimes u_m)+a\nabla \rho_m^\gamma
=\mu\Delta u_m+\widetilde\mu \nabla(\nabla\cdot u_m)-(\Delta d_m+|\nabla d_m|^2d_m)\cdot\nabla d_m
\ {\rm{in}}\ Q_1,$$
which, after sending $m\rightarrow \infty$ and applying (\ref{bd1}), (\ref{bd2}), (\ref{bd3}), and Claim 3 below,
implies
$$\nabla\rho_\infty^\gamma=0 \ \ \ {\rm{in}}\ \ \ Q_1.$$
Hence $\rho_\infty$ is $x$-independent in $Q_1$.  On the other hand, since $\rho_\infty$ is a weak solution
of
$$\partial_t \rho_\infty+\nabla\cdot(\rho_\infty u_\infty)=0\ \  \ {\rm{in}}\ \ \ Q_1,$$
so that $\partial_t\rho_\infty=0$ and $\rho_\infty$ is $t$-independent in $Q_1$.
Thus $\rho_\infty$ is a constant.

It remains to show $(\rho_m, d_m)\rightarrow (\rho_\infty, d_\infty)$ in $L^\gamma(Q_1)\times L^2([0,1], H^1_{\rm{loc}}(\Omega))$. This is divided into two separate claims.

\smallskip
\noindent{\it Claim 2.} {\it $d_m\rightarrow d_\infty$ in $L^2([0,1], H^1_{\rm{loc}}(\Omega))$}. The idea is based on
the compactness Theorem \ref{th:2.2}, and the argument is similar to that given in \S5.1 and \cite{LW1} Theorem 1.3. For the convenience of readers, we sketch it here. As in \S5.1, for $\Lambda>1$ define
$$G_\Lambda=\Big\{t\in [0,1]\ \Big| \ \liminf_{m\rightarrow\infty}\int_\Omega \Big|\Delta d_m+|\nabla d_m|^2 d_m\big|^2
\le \Lambda\Big\},$$
and
$$B_\Lambda=[0,1]\setminus G_\Lambda.$$
From (\ref{bd2}), we have
\begin{equation}\label{badset_est}
\big|B_\Lambda\big|\le \Lambda^{-1}\liminf_{m\rightarrow\infty}\int_0^1\int_\Omega \Big|\Delta d_m+|\nabla d_m|^2 d_m\big|^2=0.
\end{equation}
Since $d_m$ satisfies (\ref{stationary}) for any $X\in C_0^1(\Omega)$ and $\eta\in C_0^1((0,1))$, it
is not hard to check that there exists a subset $Z\subset G_\Lambda$, with $|Z|=0$, such that for any
$t\in G_\Lambda\setminus Z$, it holds
\begin{eqnarray}\label{stationary2}
&&\int_\Omega \Big(\nabla d_m\odot\nabla d_m-\frac12|\nabla d_m|^2\mathbb I_3\Big)(t): \nabla X
=-\int_\Omega \big\langle (\partial_t d_m+u_m\cdot\nabla d_m)(t), X\cdot\nabla d_m(t)\big\rangle\nonumber\\
&&=-\int_\Omega \big\langle (\Delta d_m+|\nabla d_m|^2 d_m)(t), X\cdot\nabla d_m(t)\big\rangle \ \ \big( {\rm{by}}\ (1.1)_3 \big)
\end{eqnarray}
It is standard (see \cite{LW5}) that (\ref{stationary2}) implies that $d_m(t)$, $t\in G_\Lambda\setminus Z$,
satisfies the almost energy monotonicity
inequality (\ref{almostmono}), i.e.,
$x_0\in\Omega$ and $0<r\le R<{\rm{d}}(x_0,\partial\Omega)$,
\begin{equation}\label{am}
\Psi_R(d_m(t),x_0)\ge \Psi_r(d_m(t),x_0)+\frac12\int_{B_R(x_0)\setminus B_r(x_0)}|x-x_0|^{-1}\big|\frac{\partial d_m(t)}{\partial |x-x_0|}\big|^2,
\end{equation}
where
$$\Psi_r(d_m(t),x_0)=\frac1{r}\int_{B_r(x_0)}\big(\frac12|\nabla d_m|^2(t)-\langle (x-x_0)\cdot\nabla d_m(t), \tau_m(t)\rangle\big)
+\frac12\int_{B_r(x_0)}|x-x_0||\tau_m(t)|^2,
$$
and
$$\tau_m(t)=(\Delta d_m+|\nabla d_m|^2 d_m)(t).
$$
From the definition of $G_\Lambda$, we have
\begin{equation}
\label{tn_bd}
\big\|\tau_m(t)\big\|_{L^2(\Omega)}\le \Lambda, \ \forall t\in G_\Lambda\setminus Z.
\end{equation}
From (\ref{bd1}), we see
\begin{equation}\label{bd4}
E(d_m(t))=\frac12\int_\Omega|\nabla d_m(t)|^2\le C({\bf E}(0)).
\end{equation}
Note also that
\begin{equation}\label{im_con}
d_m^3(x,t)\ge 0 \ {\rm{a.e.}}\ x\in\Omega, \ \forall\ t\in G_\Lambda\setminus Z.
\end{equation}
From (\ref{am}), (\ref{tn_bd}), (\ref{bd4}), and (\ref{im_con}), we conclude
that $\{d_m(t)\}_{m\ge 1}\subset {\bf Y}(C({\bf E}(0)), \Lambda, 0; \Omega)$ for any $t\in G_\Lambda\setminus Z$.
Hence, by Theorem \ref{th:2.2}, we have that $\{d_m\}_{m\ge 1}$ is bounded in $H^2_{\rm{loc}}(\Omega,\mathbb S^2)$
and precompact in $H^1(\Omega,\mathbb S^2)$.

Since $$\partial_t d_m=-u_m\cdot\nabla d_m+(\Delta d_m+|\nabla d_m|^2d_m)\in L^2([0,1], L^\frac32(\Omega))
+L^2([0,1], L^2(\Omega)),$$
and
$$\sup_{m\ge 1}\Big\|\partial_t d_m\Big\|_{L^2([0,1], L^\frac32(\Omega))
+L^2([0,1], L^2(\Omega))}\le C.
$$
We can apply Aubin-Lions' lemma, similar to \S5.1, to conclude that
for any open set $\widetilde\Omega\subset\subset\Omega$,
after taking a subsequence, there holds
\begin{equation}\label{conv1}
\lim_{m\rightarrow\infty}\Big\|\nabla (d_m-d_\infty)\Big\|_{L^2(\widetilde\Omega\times (G_\Lambda\setminus Z))}
=0.
\end{equation}
On the other hand, by (\ref{bd1}), we have
\begin{equation}\label{conv2}
\sup_{m\ge 1}\Big\|\nabla (d_m-d_\infty)\Big\|_{L^2(\widetilde\Omega\times (B_\Lambda\cup Z))}
\le C({\bf E}(0))\big|B_\Lambda\cup Z\big|=0.
\end{equation}
Putting (\ref{conv1}) and (\ref{conv2}) together yields
\begin{equation}\label{conv3}
\lim_{m\rightarrow\infty}\Big\|\nabla (d_m-d_\infty)\Big\|_{L^2(\widetilde\Omega\times (0,1))}=0.
\end{equation}
Claim 2 follows from (\ref{conv3}).

\medskip
\noindent{\it Claim} 3. {\it $\rho_m\rightarrow \rho_{0,\infty}$ in $L^\gamma(Q_1)$.}
To show this claim, first observe that by the same lines of argument in \S5.3
with $(\rho_\epsilon, u_\epsilon, d_\epsilon)$ replaced by $(\rho_m, u_m, d_m)$,
we can obtain that there exist $\theta>0$ and $C>0$ independent of $m$ such that
\begin{equation}\label{est_rho}
\int_0^1\int_\Omega \rho_m^{\gamma+\theta}\le C, \ \forall m\ge 1.
\end{equation}
From (\ref{est_rho}), we may assume that
\begin{equation}\label{conv_rho}
\rho_m^\gamma\rightharpoonup
\overline{\rho_{\infty}^\gamma}
\ {\rm{in}}\ L^{p_1}(Q_1), \ 1<p_1\le \frac{\gamma+\theta}{\gamma}(Q_1).
\end{equation}
There are two methods to prove that  $\overline{\rho_\infty^\gamma}=\rho_\infty^\gamma$
a.e. in $Q_1$ and $\rho_m\rightarrow\rho_\infty$ in $L^\gamma(Q_1)$:
the first is to repeat the same lines of arguments given by
\S5.5, \S5.6, and \S5.7 with $(\rho_\epsilon, u_\epsilon, d_\epsilon)$ replaced by
$(\rho_m, u_m, d_m)$; and the second is to apply the div-curl lemma, similar to
\cite{FP} Proposition 4.1. Here we sketch it. For simplicity, assume the pressure coefficient
$a=1$.
Let Div and Curl denote the divergence and curl operators in
$Q_1$. As pointed out by \cite{FNP1} Remark 1.1, (\ref{1.7}) also holds
for $b(\rho_m)=G(\rho_m^{\gamma})$ when $G(z)=z^\alpha$, with
$$\displaystyle 0<\alpha<\min\big\{\frac{1}{2\gamma},
\frac{\theta}{\theta+\gamma}\big\}.$$
Using the equation (\ref{1.7}), one can check that
$${\rm{Div}}\big[0,0, 0, G(\rho_m^{\gamma})\big]
\ {\rm{is\ precompact\ in}}\ W^{-1, q_1}(Q_1)$$
for some $q_1>1$.

While, using the equation $(\ref{1.1})_2$ and (\ref{bd1}), one can check
$${\rm{Curl}}\big[0, 0, 0, \rho_m^\gamma\big]
\ {\rm{is\ precompact\ in}}\ W^{-1, q_2}(Q_1)$$
for some $q_2>1$.

Assume
$$G(\rho_m^{\gamma})\rightharpoonup  \overline{G(\rho_\infty^{\gamma})}
\ {\rm{in}}\ L^{p_2}(Q_1),$$
and
$$G(\rho_m^\gamma)\rho_m^\gamma\rightharpoonup
\overline{G(\rho_\infty^\gamma) \rho_\infty^\gamma}
\ {\rm{in}}\ L^r(Q_1),$$
with
$$\displaystyle p_2=\frac{1}{\alpha}, \ \displaystyle\frac1{r}=\frac{1}{p_2}+\frac{1}{p_1}.$$
Then by the div-curl lemma we conclude that
$$\overline{G(\rho_\infty^\gamma) \rho_\infty^\gamma} =
\overline{G(\rho_\infty^\gamma)}\ \overline{\rho_\infty^\gamma}
$$
As $G$ is strictly monotone, this implies $\overline{G(\rho_\infty^\gamma)}=G\Big(\overline{\rho_\infty^\gamma}\Big)$.
Since $L^{p_2}$ is uniformly convex, this implies that the convergence in (\ref{conv_rho})
is strong in $L^1(Q_1)$. Hence we have that
$$\rho_m\rightarrow\rho_\infty \ {\rm{in}}\ L^\gamma(Q_1).$$
Since
$\displaystyle\int_\Omega\rho_m(t)=\int_\Omega \rho_0$ for $0<t<1$
and $\rho_\infty$ is constant, it follows that
$\displaystyle\rho_\infty\equiv\frac{1}{|\Omega|}\int_\Omega \rho_0\ (:=\rho_{0,\infty})$.

From claim 2, claim3, and (\ref{bd1}), we can apply Fubini's theorem to conclude that
there exists $t_m\in (m, m+1)$ such that as $m\rightarrow\infty$,
$$\big(\rho(t_m), d(t_m)\big)\rightarrow \big(\rho_{0,\infty}, d_\infty\big)
\ {\rm{in}}\ L^\gamma(\Omega)\times H^1_{\rm{loc}}(\Omega,\mathbb S^2),$$
and
$$\big\|u(t_m)\big\|_{H^1(\Omega)}\rightarrow 0.$$
Hence by Sobolev's embedding theorem we have that
$u(t_m)\rightarrow 0$ in $L^p(\Omega)$ for any $1<p<6$. The proof is now complete.
\qed

\bigskip
\bigskip
\noindent{\bf Acknowledgements}.
{Lin is partially supported by NSF of China (Grant 11001085, 11371152) and 973 Program (Grant 2011CB808002). Lai is partially supported by NSF of China (Grants 11201119 and 11126155).  Both Lin and Lai are also partially supported by the Chinese Scholarship Council.
Wang is partially supported by NSF grants 1001115
and 1265574, and NSF of China grant 11128102. The work was
completed while both  Lin and  Lai were visiting Department of Mathematics, University of Kentucky. Both of them
would like to thank the Department for its hospitality and excellent research environment. }

\end{document}